\documentclass[11pt]{article}    
\usepackage[margin={.8in}]{geometry}

\usepackage{checkend}
\usepackage{graphicx}

\usepackage{subcaption}
\usepackage{microtype}
\usepackage[shortlabels]{enumitem}
\usepackage[dvipsnames]{xcolor}
\usepackage{stackengine}
\usepackage{stmaryrd}
\SetSymbolFont{stmry}{bold}{U}{stmry}{m}{n}

\usepackage[utf8]{inputenc}
\usepackage[T1]{fontenc}
\usepackage{lmodern}

\usepackage{amsmath}
\usepackage{amsfonts}
\usepackage{amssymb}
\usepackage{amsthm}
\usepackage{mathtools}
\usepackage{color,soul}
\usepackage{bm}
\usepackage{bbm}
\def\bbm{\mathbbm}
\usepackage{silence}
\usepackage{hyperref}
\hypersetup{
    colorlinks=true,
    linkcolor=blue,
    filecolor=magenta,
    urlcolor=cyan,
    pdftitle={Overleaf Example},
    pdfpagemode=FullScreen,
    }
\usepackage{cleveref}
\usepackage{autonum}
\allowdisplaybreaks
\numberwithin{equation}{section}

\usepackage{accents}

\newtheorem{theorem}{Theorem}[section]
\newtheorem{lemma}[theorem]{Lemma}
\newtheorem{proposition}[theorem]{Proposition}

\newtheorem{corollary}[theorem]{Corollary}

\newtheorem{definition}[theorem]{Definition}

\newtheoremstyle{remarkstyle}
  {}            
  {}            
  {\upshape}    
  {}            
  {\bfseries}   
  {.}           
  {5pt plus 1pt minus 1pt} 
  {}            

\theoremstyle{remarkstyle}
\newtheorem{remark}[theorem]{Remark}

\newcounter{alphaassumption}
\newtheorem{assume}[alphaassumption]{Assumption}

\newenvironment{keywords}{
    \vspace{1em}
    \noindent\textbf{Keywords:}
    \small
    \itshape
}{}

\newenvironment{msc}{
    \vspace{1em}
    \noindent\textbf{Mathematics Subject Classification (2020).}
    \small
}{}

\newenvironment{acknowledgments}{
    \vspace{1em}
    \noindent\textbf{Acknowledgments:}
    \small
}{}

\def\mathbf{\bm}
\def\T{\trans}
\def\emptyset{\varnothing}

\def\whp{{whp}}

\def\E{\mathbb E}
\def\R{\mathbb R}
\def\N{\mathbb N}
\def\P{\mathbb P}

\def\Var{\mathrm{Var}}
\def\phi{\varphi}
\def\epsilon{\varepsilon}
\def\trans{\top}

\def\sgn{\operatorname{sgn}}

\DeclareMathOperator*{\argmax}{\operatorname{argmax}}

\title{The $\ell_r$-Levy-Grothendieck problem and $r\rightarrow p$ norms of Levy matrices}
\author{Kavita Ramanan and Xiaoyu Xie \\
Brown University}
\date{}                           

\begin{document}
\maketitle
\begin{abstract}
  Given an $n\times n$ matrix $A_n$ and $1\leq r, p \leq\infty$, we consider the following
  quadratic optimization problem referred to as the $\ell_r$-Grothendieck problem: 
\begin{align}
\begin{split}
    M_r(A_n)\coloneqq \max _{\boldsymbol{x} \in \mathbb{R}^n:\|\boldsymbol{x}\|_r \leq 1} \boldsymbol{x}^{\top} A_n \boldsymbol{x},
\end{split}
\end{align}
as well as the $r\rightarrow p$ operator norm of the matrix $A_n$, defined as 
\begin{align}
\begin{split}
    \|A_n\|_{r \rightarrow p}\coloneqq \sup _{\boldsymbol{x} \in \mathbb{R}^n:\|\boldsymbol{x}\|_r \leq 1}\|A_n \boldsymbol{x}\|_p,
\end{split}
\end{align}
where $\|\boldsymbol{x}\|_r$ denotes the $\ell_r$-norm of the vector $\boldsymbol{x}$. 
Both quantities arise in a variety of fields including theoretical computer science, statistical mechanics and asymptotic convex geometry, and are known to be hard to compute or even efficiently approximate in various parameter regimes.  This work focuses on identifying the high-dimensional asymptotics of these quantities when $A_n$ are symmetric random matrices with independent and identically distributed (i.i.d.) heavy-tailed upper-triangular entries with index $\alpha$.  We consider two regimes.  When $1\leq r\leq 2$ (respectively,  $1\leq r\leq p$) and $\alpha\in(0,2)$, we show that a suitably scaled version of $M_r(A_n)$ (respectively, $\|A_n\|_{r\rightarrow p}$) converges to a Fr\'echet distribution. On the other hand, when $2< r<\infty$ (respectively, $1\leq p< r$), we show there exists 
$\alpha_*=\alpha_*(r,p)\in(1,2)$ such that for every $\alpha\in(0,\alpha_*)$, suitably scaled versions of $M_r(A_n)$ and $\|A_n\|_{r\rightarrow p}$ converge to a limit that is a certain  power of a stable distribution.  Furthermore, we show that for each optimization problem, the constant  $\alpha_*$ is tight in the sense that  there exists  $\bar\alpha_*>\alpha_*$ such that whenever $\alpha\in(\alpha_*,\bar\alpha_*)$  the same asymptotic convergence holds only when the matrix entries are centered;   when the matrix entries have non-zero mean,  the corresponding quantity  converges to a different limit, and only after  a different normalization and  additional centering. The analysis uses a combination of tools from the theory of heavy-tailed distributions, the nonlinear power method and concentration inequalities. 
Our results shed light on the differences in the asymptotic behavior of operator norms of heavy-tailed matrices, when compared to both spectral asymptotics for Levy matrices, as well as recent work on  norm  asymptotics for light-tailed random matrices.  
Furthermore, as a corollary of our results we also characterize the limiting ground state of the Levy spin glass model when $\alpha \in (0,1)$.
\end{abstract}

\begin{msc}
Primary: 60B20, 60E07; Secondary: 60F05, 15A60, 46B09 
\end{msc}


\section{Introduction}

\subsection{Motivation and Context}
\label{subs-motivation}

Given an $n\times n$ matrix $A$, consider the following two nonlinear optimization problems: 
\begin{enumerate}
\item For $1\leq r\leq \infty$, the $\ell_r$-Grothendieck problem concerns the quantity
\begin{equation}\label{opt1}
M_r(A)\coloneqq \max _{\bm x \in \mathbb{R}^n:\|\bm x\|_r \leq 1}\bm x^\T A\bm x,\tag*{\textsf{OPT}\,1}
\end{equation}
where $\|\bm x\|_r\coloneqq \left(\sum_{i=1}^n\left|x_i\right|^r\right)^{1/r}$ denotes the usual $\ell_r$-norm of the vector $\bm x$;  
\item For $1\leq r,p\leq \infty$, the $r\rightarrow p$ operator norm of $A$ is defined as
\begin{equation}\label{opt2} 
\|A\|_{r\rightarrow p}\coloneqq \sup_{\bm x \in \mathbb{R}^n:\|\bm x\|_r\leq1}\|A\bm x\|_p.\tag*{\textsf{OPT}\,2}
\end{equation}
\end{enumerate}
For different choices of  $r$ and $p$, $M_r(A)$ and $\|A\|_{r\rightarrow p}$ correspond to key quantities of interest arising in different fields.
Specifically, $M_2(A)$ represents the largest eigenvalue of the symmetric matrix $(A+A^\T)/2$, where $A^\T$ denotes the transpose of $A$. In the context of clustering, the $\ell_r$-Grothendieck problem for $r=2$ and $r=\infty$ is related to spectral partitioning \cite{donath1973lower,fiedler1973algebraic} and correlation clustering \cite{CW04}, respectively, and  for general $r\in(2,\infty)$,  can be viewed as a smooth interpolation between these two clustering measures. 
When $A$ is diagonal free, and has off-diagonal entries that are independent and identically distributed (i.i.d.) standard Gaussian random variables, $M_r(A)$ is equal to the ground state energy of the Sherrington-Kirkpatrick spin glass model with spins coming from the $\ell_r$-sphere and couplings given by the entries of $A$ (see the discussion in \cite{ChenSen20} after (1.2) as well as the discussion in \Cref{s:main} prior to \Cref{cor:ground}).

Turning to the $r\rightarrow p$ operator norm, $\|A\|_{2\rightarrow 2}$ is equivalent to the largest singular value of $A$. 
In the scenario where $r<p$ and $A$ represents the normalized adjacency matrix of a graph (or, more generally, of a Markov operator), upper bounds on $\|A\|_{r\rightarrow p}$ are recognized as hypercontractive inequalities, which  are instrumental in demonstrating rapid mixing properties of  associated random walks (see \cite[Section~2.2]{gine1997lectures}). 
Additionally, the computation of $\|A\|_{2\rightarrow 4}$ is closely related to  small-set expansion and the injective tensor norm, both of which are important in theoretical computer science \cite{barak2012hypercontractivity,brandao2015estimating}.
Other application domains where the $\|A\|_{r\rightarrow p}$ norm arises include the development of oblivious routing schemes in transportation networks for the $\ell_p$ norm \cite{BV11,englert2009oblivious,GHR06,racke2008optimal}, and data dimensionality sketching for $r\rightarrow p$ norms, which has applications to data streaming models and robust regression \cite{KMW18}.
The special case $\|A\|_{p\rightarrow p}$ also appears in the estimation of matrix condition numbers \cite{Hig87}.

It is worth pointing out that the two optimization problems are connected; the quantity
$M_r(A)$ of \ref{opt1} is closely related to $\|A\|_{r\rightarrow p}$ in \ref{opt2} 
in the special case $p=r^*\coloneqq r/(r-1)$.  Indeed, 
given any $\bm x\in\mathbb R^n$ with $\|\bm x\|_r\leq 1$, using the well-known duality representation $\|A\bm x\|_{p}=\sup_{y\in\R^n:\left\|\bm y\right\|_{p^*}\leq 1}\bm y^\top A\bm x$ (see for example \cite[Proposition~6.13]{folland1999real}) 
it follows that 
\begin{equation}\label{eq:rtop-dual}
\begin{aligned}
    \|A\|_{r\rightarrow p}=\sup_{\bm x,\bm y\in\mathbb R^n:\|\bm x\|_{r}\leq 1,\|\bm y\|_{p^*\leq 1}}\bm y^\T A\bm x. 
\end{aligned}
\end{equation}
Thus, in the case $p = r^*$, we have $p^* = r$, and comparing \ref{opt1} with  \eqref{eq:rtop-dual}, we see  that 
\begin{equation}\label{ineq:dual} 
\begin{aligned}
\|A\|_{r\rightarrow r^*}\geq M_r(A).
\end{aligned}
\end{equation}
In  specific instances, such as when $r \geq 2$, $A$ has nonnegative entries, and $A^\T A$ is irreducible,  $M_r(A)$ is equal  to $\|A\|_{r\rightarrow r^*}$  (see \cite[Proposition~2.14]{dhara2020r}). However, in general, these two quantities need not coincide. 

The computational complexity of both optimization problems 
has attracted immense attention in the theoretical computer science community. 
For $r\in(2,\infty)$, while there exists a polynomial time algorithm to approximate $M_r(A)$ within a factor of $\xi_r^2$, where $\xi_r\coloneqq\E\left[|Z|^r\right]^{1/r}$, with $Z$ being a standard normal random variable,   
no polynomial-time algorithm can approximate $M_r(A)$ within a factor strictly less than $\xi_r^2$ unless P$=$NP \cite{khot2011grothendieck}.  
When $A$ is a nonnegative matrix and $r\geq p$, efficient algorithms such as the nonlinear power method by Boyd \cite{boyd1974power} can compute $\|A\|_{r\rightarrow p}$ in polynomial time.  
However, for general matrices $A$, when $r\geq 2 \geq  p \geq 1$, it is NP-hard to approximate $\|A\|_{r\rightarrow p}$ to within any factor smaller than $1/(\xi_{p}\xi_{r/(r-1)})$ \cite{bhattiprolu2018inapproximability}. And when either  $r\geq p$ and $2\notin [p,r]$, or $r < p$,  it is NP-hard to 
approximate  $\|A\|_{r\rightarrow p}$ to within any constant factor \cite{bhattiprolu2018inapproximability}. 
The  challenges associated with computing $M_r(A)$ and $\|A\|_{r\rightarrow p}$ lead naturally to the following  questions: \\

\noindent 
{\bf Q.} Is the  computation and approximation of the  quantities  $M_r(A)$ and $\|A\|_{r\rightarrow p}$ also hard for typical matrices, such as random matrices with i.i.d. entries, modulo symmetry?  Can one characterize suitably scaled asymptotic limits of these quantities? How does the answer depend on the structure of the matrix?  

In Section \ref{subs-prior} we discuss prior work that sheds some light on this question, then in Section \ref{subs-summaryofres} summarize our main results, and in Section \ref{s:notat} introduce common notation used throughout the paper. 
 
\subsection{Discussion of Related Prior Work}
\label{subs-prior}

While classical work on random matrices has mostly focused on the spectral setting, 
recent years have witnessed a resurgence in interest in operator norms of random matrices.
 One line of inquiry considers non-asymptotic bounds. 
In early work \cite{bennett1975norms}, Bennett obtained  
an upper bound for the expected $2 \to p$ norm for 
matrices with independent, mean-zero entries lying in $[-1,1]$ when $p > 2$. 
He subsequently extrapolated these results to obtain bounds on 
$r \to p$ norms in the full range  $1\leq r,p\leq \infty$.
\cite{bennett1977schur}.  
More recently, Lata{\l}a and Strzelecka \cite{latala2023chevet,latala2024operator} established matching two-sided nonasymptotic bounds for the expected $r \to p$ norm, with $r,p\in[1,\infty]$, 
for matrices with  i.i.d. entries drawn from a large class of light-tailed distributions. 
For matrices with independent (but not necessarily identical)
mean-zero Gaussian entries,  
two-sided bounds that  match up to a
logarithmic factor depending on the dimension of the matrix 
were derived by Adamczak et al  \cite{adamczak2023norms}. 
It was also  conjectured in \cite{adamczak2023norms} that this logarithmic factor could be removed to yield a dimension-free  estimate of the $r\rightarrow p$ operator norm of inhomogeneous Gaussian random matrices.

The other line of inquiry focuses on high-dimensional asymptotic
limits of  $\|A\|_{r\rightarrow p}$ and $M_r(A)$.  
When $r = 2$,  $M_2(A)$ is equal to the largest eigenvalue of $\frac{A+A^\T}{2}$, and falls 
into the well-studied classical setting of spectral norms,  
When $A$ has i.i.d. centered Gaussian entries,  the 
asymptotic fluctuations of $M_2(A)$ are well known to be governed by the celebrated Tracy-Widom distribution \cite{tracy1996orthogonal} (see also the survey paper \cite{erdHos2011universality}). 
On the other hand, for matrices with i.i.d. nonnegative entries with bounded support and positive mean, the fluctuations of $M_2(A)$ where shown to be Gaussian by F\"{u}redi and Koml\'{o}s. 
\cite{FurKom81}.  A generalization of the latter result to operator norms was established by Dhara et al \cite{dhara2020r} for matrices $A$ that are symmetric, nonnegative random matrices with i.i.d. light-tailed upper triangular entries. Specifically, leveraging 
the (nonlinear) power method of Boyd \cite{boyd1974power}, they identified the almost sure limit of suitably normalized $\|A\|_{r\rightarrow p}$ norms when $r\geq p$ as well as $M_r(A)$ when $r \geq 2$, and also showed that the fluctuations are Gaussian and identified the variance. 
On the other hand, when $A$ is an $n \times n$ matrix composed of $n^2$ i.i.d. standard Gaussian random variables, Chen and Sen 
 \cite{ChenSen20} used Chevet's inequality to examine 
the almost sure limit of an appropriately scaled version of $M_r(A)$.  They discovered that the scaling behavior of $M_r(A)$ undergoes a phase transition at the ``spectral threshold'' $r=2$:  
for $r\in(1,2)$, $M_r(A)$ scales as $n^{(r-1)/r}$, and for $r\in[2,\infty]$, $M_r(A)$  scales as $n^{3/2-2/r}$.  Moreover, they also characterized the almost sure limit of  $M_r(A)$ (suitably normalized) for all $r\geq 1$, 
in particular showing that in the sub-regime $r > 2$  it can be characterized via a variant of the Parisi formula for the free energy of the Sherrington-Kirkpatrick spin glass model. 
Going beyond almost sure limits, 
characterization of the fluctuations of $M_r(A)$ for matrices with i.i.d. centered Gaussian entries  (beyond the spectral case  $r = 2$)  remains an interesting open problem. 

As can be seen from the discussion above, the majority of work on general operator norms of random matrices thus far has focused on the case of light-tailed entries. 
Work on random matrices with heavy-tailed entries has been essentially restricted to the spectral setting:  
 the largest eigenvalue of $A$ is known to follow a Fr\'{e}chet distribution when the index $\alpha$ of the tail lies in the interval  $(0,4)$ \cite{Sosh04,AAP09}.  
 Thus,   in contrast to their light-tailed counterparts, random matrices with heavy-tailed entries exhibit very different properties.

\subsection{Summary of Results and Open Problems}
\label{subs-summaryofres}

Our work takes a first step towards understanding the asymptotics of non-spectral norms and the $\ell_r$-Grothendieck problem in the heavy-tailed setting.  In contrast to the current state of art for Gaussian matrices, it is possible to identify both first and second-order asymptotics in certain regimes. 
Let $A_n=(a_{ij})_{i,j\in[n]}$ be a symmetric $n \times n$  matrix consisting of i.i.d. heavy-tailed upper triangular entries with  index $\alpha\in(0,2)$; see \Cref{d:heavy}.  
Our main results can be broadly categorized into two distinct regimes: 
\begin{enumerate}
\item Suppose $1\leq r\leq 2$ for $M_r(A_n)$ and $1\leq r\leq p$ for $\|A_n\|_{r\rightarrow p}$, and $\alpha\in(0,2)$. Then we show that suitably scaled versions of the quantities $M_r(A_n)$ and $\|A_n\|_{r\rightarrow p}$ 
converge weakly, as $n \rightarrow \infty$,  to a Frech\'et distribution  (see \Cref{thm:gro1} and \Cref{thm:rtop1}).  
\item Suppose $2<r\leq\infty$ for $M_r(A_n)$ and $1\leq p<r\leq \infty$ for $\|A_n\|_{r\rightarrow p}$, 
and $\alpha$ lies in $(0,\alpha_*)$ for a certain explicit constant $\alpha_*=\alpha_*(r)$ or $\alpha_*=\alpha_*(r,p)$ lying in the interval $(1,2)$. 
Then 
both quantities, suitably scaled, converge to a certain power of a  stable distribution  (see \Cref{thm:gro2} and \Cref{thm:rtop2}).
\end{enumerate}
Furthermore, we establish the following results that  
show that the choice of $\alpha_*$ in Case 2 is in a sense optimal: 
\begin{enumerate}
\item[3.] Suppose that either $2<r<\infty$ for $M_r(A_n)$ or $1\leq r<p\leq \infty$ for $\|A_n\|_{r\rightarrow p}$, with the corresponding constant $\alpha_*\in(1,2)$ as in Case 2.
  Then there exist (explicit) constants  $\bar\alpha_{*,1}$
  and   $\bar\alpha_{*,2}$ that both lie in the interval $(\alpha_{*},2)$, and depend on $r$, or both $r$ and $p$ for \ref{opt2}, 
   such that the following properties are satisfied:
 \begin{itemize}
\item If the entries of $A$ have zero mean and $\alpha\in(\alpha_*,\bar\alpha_{*,1})$, then the same convergence results hold as in Case 2  (see \Cref{thm:gro2'} and \Cref{thm:rtop2'}), although the proof slightly differs. 
\item On the other hand, if the entries of  $A$ have positive mean for $M_r(A)$, or nonzero mean for $\|A\|_{r\rightarrow p}$, and $\alpha\in(\alpha_*,\bar\alpha_{*,2})$,
  then the asymptotic fluctuations of  $M_r(A)$ and $\|A\|_{r\rightarrow p}$, after a different normalization and an additional centering,  converge to stable distributions (see \Cref{thm:gro3} and \Cref{thm:rtop3}).
\end{itemize}
\end{enumerate}
Also, as an immediate consequence
of our results,  
\begin{itemize} 
\item[4.] 
we characterize the ground state of the Levy spin glass model  from
statistical physics when $\alpha \in (0,1)$ (see \Cref{cor:ground}). 
\end{itemize}

As elaborated in  \Cref{s:main}, 
the proofs proceed by decomposing the matrix into a sum of matrices with small, intermediate and large entries (in absolute value), suitably defined, and to then show that only the large entries dominate and characterize their limit.  This is done by exploiting a combination of tools from the theory of heavy tails, random matrix theory, concentration estimates as well as analysis of the nonlinear optimization problems  to arrive at suitable upper and lower bounds on the $r \rightarrow p$ norm of the dominant matrix in the decomposition that is then shown to coincide in the asymptotic limit.  The latter estimates are quite subtle and distinguish the analysis  from that in the spectral setting.

Our results on the $\ell_r$-Levy-Grothendieck problem in the heavy-tailed setting are summarized in \Cref{fig:combined}(a) and \Cref{fig:combined}(b) for the case of centered and non-centered entries, respectively.  
An interesting open problem is to identify the asymptotic fluctuations in the orange regimes in these
two figures.  
This regime consists of intervals of $\alpha$ exceeding a specific  threshold where  the tails of the entries are not as heavy,  with this threshold decreasing with increase in  $r$, though always remaining above $1$.   In this regime, the fluctuations are  determined by both small and large (in absolute value) entries rather than just the latter, making the analysis more subtle. 
In fact, the influence of  smaller entries on fluctuations already starts to manifest
itself in the third set of results described above, in the form of the difference in scaling factors for non-centered entries. 
The fact that the regime with lighter tails is more difficult 
is  consistent with the fact that,  as mentioned in Section \ref{subs-prior}, even for centered Gaussian i.i.d. entries, the asymptotic fluctuations of $M_r(A_n)$ are not
well understood beyond the spectral case $r = 2$.  
It is also worth mentioning that the change in scalings for different $\alpha$ regimes and the difference between  centered and non-centered entries are not observed in the spectral case $r=2$ in the regime  $\alpha \in (0,2)$, 
where, as mentioned at the end of \Cref{subs-prior}, the properly scaled largest eigenvalue $M_2(A_n)$ always converges to the Fr\'echet distribution. Although, it is worth mentioning that such a this distinction in  the asymptotics of centered and non-centered matrices does manifest itself for the study of a particular quantity in the spectral regime when $\alpha \in (2,4)$ \cite{AAP09}. 

\begin{figure}[ht!]
    \centering
    \begin{subfigure}[b]{0.45\textwidth}
        \centering
        \includegraphics[width=\textwidth]{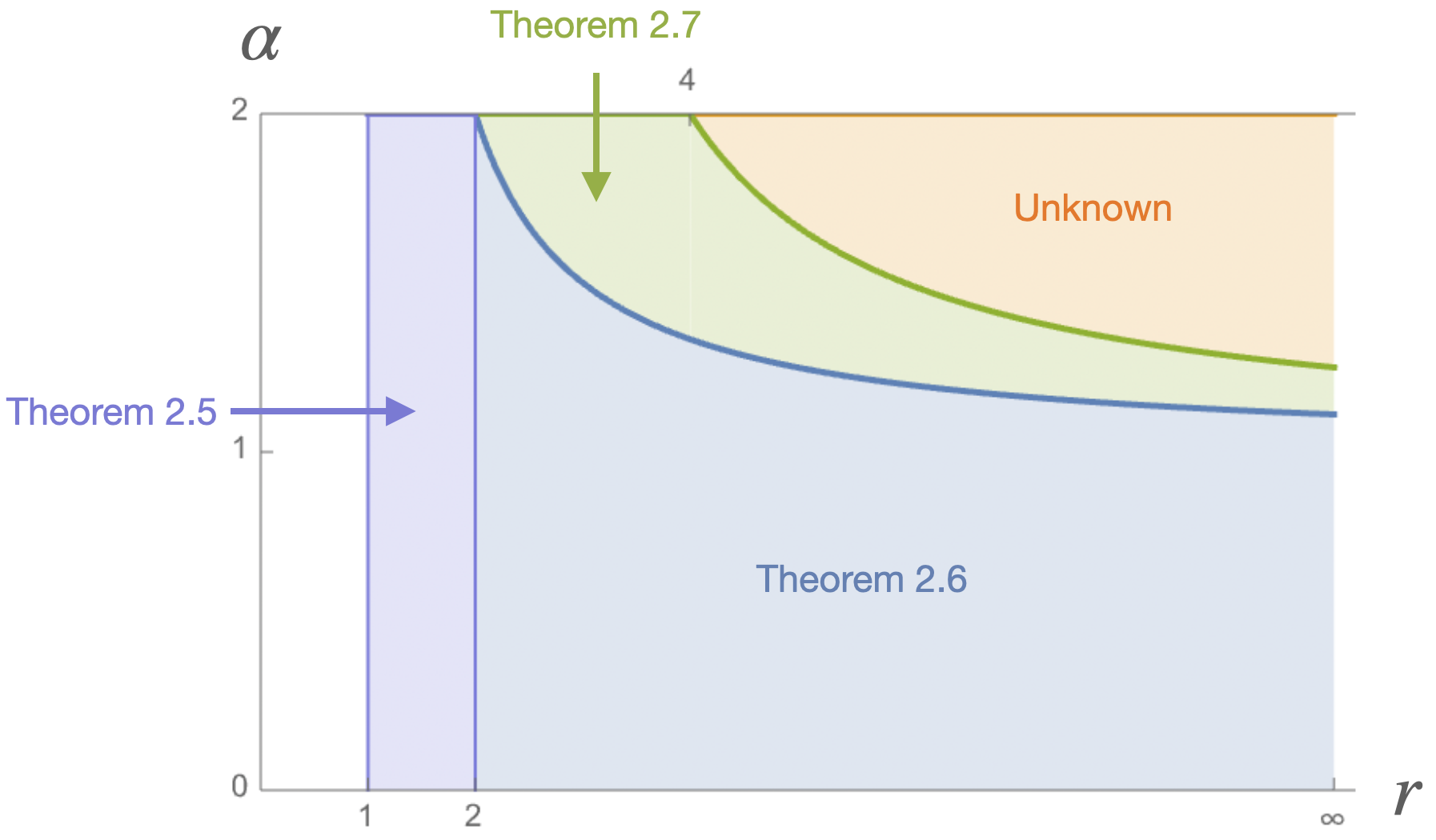}
        \caption{Entries with zero mean}
        \label{fig:sub1}
    \end{subfigure}
    \hfill
    \begin{subfigure}[b]{0.45\textwidth}
        \centering
        \includegraphics[width=\textwidth]{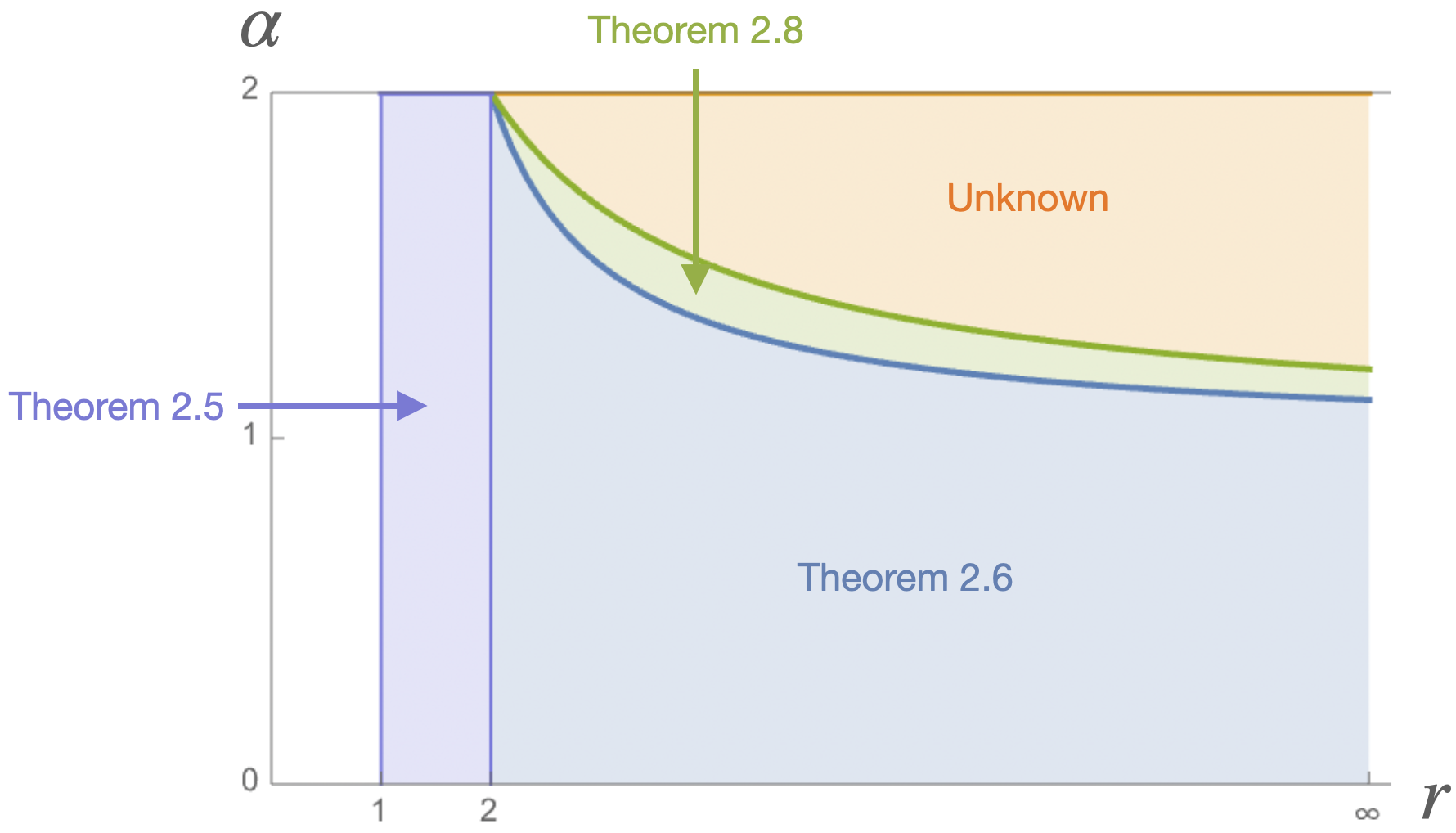}
        \caption{Entries with nonzero mean}
        \label{fig:sub2}
    \end{subfigure}
    \caption{}
    \label{fig:combined}
\end{figure}

\subsection{Common Notation}
\label{s:notat}

For any positive integer $n$, let $[n]$ denote the set $\{1,\ldots,n\}$.
For a sequence of random variables $\{Y_n\}_{n\in\N}$ and a random variable $Y$, we denote the convergence in distribution of $Y_n$ to $Y$ by $Y_n\xrightarrow{(d)}Y$. 
Given a probability space $\left(\Omega,\mathcal F,\mathbb P\right)$, the events $E_n\in\mathcal F,n\in\N$ are said to occur \emph{with high probability} (\whp)\ if $\P(E_n)\rightarrow 1$ as $n\rightarrow\infty$. 
For a sequence of random variables $(X_n)_{n\in\N}$ and positive deterministic constants $(k_n)_{n\in \N}$:
\begin{enumerate}
\item We say $X_n=O(k_n)$ \whp\ if there exists a constant $C>0$ such that the events $\{|X_n|\leq Ck_n\}_{n\in\N}$ occur \whp.  We say $X_n=O_\delta(k_n)$ \whp\ if for any $\delta>0$ the events $\{|X_n|\leq n^\delta k_n\}_{n\in\N}$ occur \whp.
\item We say $X_n=o(k_n)$ \whp\ if $k_n^{-1}X_n$ converges to $0$ in probability.
\item We say $X_n=\Theta_\delta(k_n)$ \whp\ if for any $\delta>0$ the events $\{n^{-\delta}k_n\leq |X_n|\leq n^\delta k_n\}_{n\in \N}$ occur \whp.
\end{enumerate}
For a sequence of deterministic quantities $\{D_n\}_{n\in\N}$ and positive deterministic constants $(k_n)_{n\in \N}$:
\begin{enumerate}
\item We say $D_n=O(k_n)$ if there exists a constant $C>0$ such that $\limsup_{n\rightarrow\infty}k_n^{-1}|D_n|\leq C$.  We say $D_n=O_\delta(k_n)$ if for any $\delta>0$, $\lim_{n\rightarrow \infty}n^{-\delta}k_n^{-1}D_n=0$.
\item We say $D_n=o(k_n)$ if $\lim_{n\rightarrow\infty}k_n^{-1}D_n=0$.
\item We say $D_n=\Theta_\delta(k_n)$ if for any $\delta>0$, $\lim_{n\rightarrow\infty}n^{-\delta}k_n^{-1}D_n=0$ and $\lim_{n\rightarrow\infty}n^{\delta}k_n^{-1}|D_n|=\infty$.
\end{enumerate}
Also, given nonnegative functions $f,g$ defined on $[0,\infty)$, we say $f\ll g$ (respectively, $f\gg g$) if $f(x)/g(x)\rightarrow0$ (respectively, $g(x)/f(x)\rightarrow0$) as $x\rightarrow \infty$.

Given an $n\times n$ matrix $A$, we denote its transpose by $A^\T$ and the matrix obtained by taking the absolute value of each entry in $A$ by $|A|$. Given a vector $\bm v=(v_1,\ldots,v_n)\in\mathbb R^n$ with positive entries and for $p\in\mathbb R$, we define $\bm v^p\coloneqq (v_1^p,\ldots, v_n^p)$. 
 The $p$-norm of a vector $\bm x=\left(x_i\right)_{i\in[n]}\in\mathbb R^n$ is defined as 
\begin{align}
\begin{split}
    \|\bm x\|_p\coloneqq \left(\sum_{i=1}^n\left|x_i\right|^p\right)^{\frac{1}{p}}.
\end{split}
\end{align}
For $p\in(1,\infty)$, we denote its H\"older conjugate by $p^*\coloneqq p/(p-1)$. Additionally, we write $1^*=\infty$ and $\infty^*=1$.

We use $\sgn:\mathbb R\rightarrow \mathbb R$ to denote the sign function, where $\sgn(t)=-1,1$ and $0$ for $t<0,t>0$ and $t=0$, respectively.
For $q>1,\ t \in \mathbb{R}$ and $\bm x=(x_1,\ldots,x_n) \in \mathbb{R}^{n}$, we define $\Psi_q:\mathbb R^n\rightarrow\mathbb R^n$ as follows: for $i\in[n]$,
\begin{equation}\label{def:psi}
\left(\Psi_{q}(\bm x)\right)_i\coloneqq\psi_{q}\left(x_{i}\right),\quad\text{where}\ \ 
\psi_{q}(t)\coloneqq|t|^{q-1} \operatorname{sgn}(t).
\end{equation}

\begin{acknowledgments}
K.R. and X.X. were supported in part by the Office of Naval Research under the Vannevar Bush Faculty Fellowship NR-N0014-21-1-2887 and 
by the National Science Foundation under grants DMS-1954351 and DMS-2246838.
\end{acknowledgments}
\section{Main results}\label{s:main}
In \Cref{ss:prelim} we state our main  assumption on the random matrix after introducing some basic definitions related to heavy-tailed distributions. In \Cref{ss:gro-results} and \Cref{ss:rtop-results}, we present our main results concerning \ref{opt1} and \ref{opt2},  respectively, along with a discussion of the proofs. 
\Cref{ss:organization} summarizes the rest of the paper.

\subsection{Assumption on the random matrix}\label{ss:prelim}
We start with the definition of a slowly varying function. 

\begin{definition}[Slowly Varying Function]
A function $L:\mathbb (0,\infty)\rightarrow (0,\infty)$ is said to be a \emph{slowly varying function} if it satisfies
\begin{align}\label{eqn:slowly-varying}
\begin{split}
    \lim _{x \rightarrow \infty} \frac{L(t x)}{L(x)}=1,\quad \forall\, t>0.
\end{split}
\end{align}
\end{definition}
The next lemma states that slowly varying functions grow (or decay) asymptotically slower than any polynomial.
\begin{lemma}{\cite[Lemma~2.7]{nair2013fundamentals}}\label{l:slowly-varying-property}
If the function $L:(0,\infty)\rightarrow (0,\infty)$ is slowly varying, then
\begin{align}\label{eqn:slowly-varying-property}
\begin{split}
    x^{-\delta} \ll L(x) \ll x^\delta,\quad \forall\,\delta>0.
\end{split}
\end{align}
\end{lemma}
\begin{definition}[Heavy-tailed random variable]\label{d:heavy}
A random variable $X$ is said to be \textnormal{heavy-tailed} with index $\alpha>0$ if it satisfies
\begin{align}\label{ccdf}
\begin{split}
    \bar F(x)\equiv \bar F_\alpha(x)\coloneqq\P\left(|X|>x\right)= x^{-\alpha}L(x),\quad x> 0,
\end{split}
\end{align}
where $L$ is a slowly varying function. 
\end{definition}

Stable random variables are a special class of heavy-tailed random variables. 
\begin{definition}[Stable random variable]
\label{def:stable}
For $\alpha\in(0,2)$, a nondegenerate random variable $X$ is said to be $\alpha$-stable if $X=aZ+b$, where $a\neq 0$, $b\in\mathbb R$, and $Z$ is a random variable whose characteristic function takes the following form, for some $\beta\in[-1,1]$:
\begin{align}
\begin{split}
    \E \left[\exp (i t Z)\right]= \begin{cases}\exp \left(-|t|^\alpha\left[1-i \beta \tan \frac{\pi \alpha}{2}(\operatorname{sgn} t)\right]\right), & \alpha \neq 1, \\ \exp \left(-|t|\left[1+i \beta \frac{2}{\pi}(\operatorname{sgn} t) \log |t|\right]\right), & \alpha=1.\end{cases}
\end{split}
\end{align}
\end{definition}

 Throughout, we consider sequences of random matrices satisfying the following assumption.
\begin{assume}\label{assume}  There exists 
a heavy-tailed random variable with index $\alpha \in (0,2)$ such that 
for each $n \in \N$, $A_n$ is a symmetric $n \times n$ random matrix whose upper-triangular entries are i.i.d. copies of  that heavy-tailed random variable. 
\end{assume}

Fix $\alpha\in(0,2)$. Given the complementary cumulative distribution function $\bar F_\alpha$ of the absolute value of the heavy-tailed distribution with index $\alpha$,  define
\begin{equation}\label{d:b}
b_n\equiv b_{n,\alpha}\coloneqq\inf \left\{x\geq0: \bar F_\alpha(x) \leq \frac{2}{n(n+1)}\right\},\quad n\in\mathbb N.
\end{equation}
It can be shown that (see \cite[(3.8.6)]{Durrett19})
\begin{equation}\label{eqn:property-slowly-varying}
\bar F_\alpha(b_{n,\alpha})\cdot\frac{n(n+1)}{2}\rightarrow1,\text{ as }n\rightarrow\infty,
\end{equation}
which implies that there exists another slowly varying function $L_o$ such that
\begin{equation}\label{eqn:order-bn}
b_{n,\alpha} = L_{o}(n) n^{2 / \alpha},
\end{equation}
For notational simplicity, we often drop the $\alpha$ subscript whenever it is clear from the context.


\subsection{The Levy-Grothendieck \texorpdfstring{$\ell_r$}{PDFstring} problem and the ground state of Levy spin glasses}\label{ss:gro-results}
 
Our first set of results concerns the $\ell_r$-Grothendieck problem (\ref{opt1}) for  sequences of random matrices satisfying \Cref{assume}. 

\begin{theorem}\label{thm:gro1}
Fix $1\leq r\leq 2$. Suppose the sequence $A_n=\left(a_{ij}\right)_{i,j\in[n]}$, $n \in \N$, satisfies \Cref{assume} with $\alpha\in(0,2)$. Then with $b_n$ as defined in \eqref{d:b}, we have
\begin{align}
\begin{split}
    \lim _{n \rightarrow \infty} \mathbb{P}\left(b_{n}^{-1}2^{\frac{2}{r}-1}M_r(A_n) \leq x\right)=\exp \left(-x^{-\alpha}\right).
\end{split}
\end{align}
\end{theorem}
The proof of \Cref{thm:gro1}, which  is given in \Cref{ss:gro1}, relies on results obtained in \Cref{s:rtop1}. 
 It proceeds by analyzing \ref{opt1} to show that when $1\leq r\leq 2$, $M_r(A_n)$ is completely determined by the largest (in absolute value) entry of $A_n$, and then using  properties of heavy-tailed matrices (Lemma \ref{lemma:typicalBehaviour}) to conclude that the limiting distribution of the properly normalized $M_r(A_n)$ resembles that of the limiting distribution of the largest entry, which is well known to be a Fr\'echet distribution with index $\alpha$  (Lemma \ref{Frechet}).   
 
 The analysis in the second regime where  $r > 2$ is more involved. 
 
\begin{theorem}\label{thm:gro2}
Fix $2<r\leq\infty$ and let $\alpha_*\equiv\alpha_*(r)\coloneqq\frac{r}{r-1}$. Suppose the sequence $A_n=\left(a_{ij}\right)_{i,j\in[n]}$, $n \in \N$, satisfies \Cref{assume} with an index $\alpha$  that 
satisfies 
\begin{align}\label{gro2-cond1}
\begin{split}
    0<\alpha<\alpha_*(r).
\end{split}
\end{align}
Then with $b_n$ as defined in \eqref{d:b}, we have
\begin{align}
\begin{split}
    b_n^{-1}M_r(A_n)\xrightarrow{(d)} \left(Y_{\frac{\alpha(r-2)}{r}}\right)^{\frac{r-2}{r}},
\end{split}
\end{align}
where $Y_{\frac{\alpha(r-2)}{r}}$ is an $\frac{\alpha(r-2)}{r}$-stable random variable.
\end{theorem}

The proof of \Cref{thm:gro2} is deferred to \Cref{ss:gro2}, and relies on refined estimates  obtained in Section \ref{s:rtop2}.   
 To deal with the heavy-tailed nature of entries, the main strategy is to first decompose $A$ based on the magnitude of the entries, in the following manner
\begin{align}
\begin{split}
    A=  A^{\mathsf{la}}+A^{\mathsf{int}}+ A^{\mathsf{sm}},
\end{split}
\end{align}
where 
$ A^{\mathsf{la}}$ contains the largest (in absolute value) entries and is overall sparse, $ A^{\mathsf{int}}$ consists of entries of intermediate magnitude and is sparse on each column and row, and the remaining matrix $ A^{\mathsf{sm}}$ that contains entries of small magnitude is dense;    see \eqref{a-split} for precise definitions.   The second ingredient is a general upper bound on the $r \rightarrow p$ norm of (and  the $M_r$ value) of any matrix in  terms of a nonlinear function of its row sums (\Cref{lem:new-upper-bound}).   On the one hand, due to the dense nature of $ A^{\mathsf{sm}}$ one can  apply concentration estimates that exploit  cancellations in the row sums, and invoke the   
upper bound to  show that 
(a suitable scaled version of ) $M_r( A^{\mathsf{sm}})$ is asymptotically negligible (\Cref{lem:omit-small}).  
On the other hand, to show that (the scaled version of) $M_r( A^{\mathsf{int}})$ is  asymptotically negligible, one instead leverages the row-wise sparseness  along with  the heavy-tailed nature of the entries  (see \Cref{lem:omit-middle}).    Finally, we show that the `extreme' sparsity of $ A^{\mathsf{la}}$ simplifies its structure, making it amenable to  explicit optimization, which in view of the stable central limit theorem yields  the result.

The next two results show that the upper bound $\alpha_*(r) = \frac{r}{r-1}$ in \Cref{thm:gro2} is in a sense tight.  
When the tails become less heavy, that is, when $\alpha$ exceeds the value $\alpha_*(r)$, the 
asymptotic behavior becomes even more subtle and both the scaling and asymptotics depend  on whether or not the entries are centered. 
As shown in  \Cref{thm:gro2'} below, in this regime the asymptotics are   the same as that  in \Cref{thm:gro2} when the entries are centered, and the proof  is also similar and given in  \Cref{ss:gro2}. 
 
\begin{theorem}\label{thm:gro2'}
Fix $2<r\leq \infty$, $\alpha_*\equiv \alpha_*(r)\coloneqq\frac{r}{r-1}$ and $\bar\alpha_{*,1}\equiv\bar\alpha_{*,1}(r)\coloneqq\min\{2,\frac{r}{r-2}\}$. Suppose the sequence $A_n=\left(a_{ij}\right)_{i,j\in[n]}$, $n \in \N$, satisfies \Cref{assume} with index $\alpha$, such that
\begin{align}\label{gro2-cond2}
\begin{split}
    \text{$(a_{ij})_{i,j\in[n]}$ are centered and}\quad\alpha_*(r)\leq \alpha<\bar\alpha_{*,1}(r).
\end{split}
\end{align}
Then with $b_n$ as defined in \eqref{d:b}, we have
\begin{align}
\begin{split}
    b_n^{-1}M_r(A_n)\xrightarrow{(d)} \left(Y_{\frac{\alpha(r-2)}{r}}\right)^{\frac{r-2}{r}},
\end{split}
\end{align}
where $Y_{\frac{\alpha(r-2)}{r}}$ is an $\frac{\alpha(r-2)}{r}$-stable random variable.
\end{theorem}
 
The reason for the need for centering in this regime can be understood as follows.
When $\alpha$ is small, that is, when \eqref{gro2-cond1} holds, the entries have ``extreme'' heavy tails, making the impact of the mean (or median for $\alpha\in(0,1]$) of the distribution, irrespective of whether or not it is zero, asymptotically negligible  when compared to the impact of the sparse collection of the largest (in absolute value) entries.  
On the other hand, 
when $\alpha\geq \alpha_*(r)$, 
the assumption of zero mean is necessary to make the influence of smaller entries negligible and still yield a sparse structure for the largest entries in the matrix.  
Indeed, as we show in \Cref{thm:gro3} below, when the entries have a strictly positive mean, one has different  asymptotics.  

\begin{theorem}\label{thm:gro3}
Fix $2<r< \infty$, $\mu>0$, $\alpha_*\equiv\alpha_*(r)\coloneqq \frac{r}{r-1}$ and $\bar\alpha_{*,2}\equiv\alpha_{*,2}(r)\coloneqq \frac{r+2}{r}$. Suppose the sequence  $A_n=(a_{ij})_{i,j\in[n]}, n \in \N,$ satisfies \Cref{assume} with $\alpha$, such that the following additional condition holds: 
\begin{align}\label{gro3-cond}
\begin{split}
    \text{$(a_{ij})_{i,j\in[n]}$ are centered and}\quad\alpha_*(r)<\alpha<\bar\alpha_{*,2}(r).
\end{split}
\end{align}
Define
\begin{align}
\begin{split}
    A_{\mu,n}\coloneqq \mu\bm 1\bm 1^\T+A_n=(\mu+a_{ij})_{i,j\in[n]}.
\end{split}
\end{align}
Then with $b_n$ as defined in \eqref{d:b}, we have
\begin{align}
\begin{split}
    b_n^{-\frac{r}{r-2}}n^{\frac{r}{r-2}-\frac{r-2}{r}}\left(M_r(A_{\mu,n})-n^{2-\frac{2}{r}}\mu\right)\xrightarrow{(d)}Y_{\frac{\alpha(r-2)}{r}},
\end{split}
\end{align}
where $Y_{\frac{\alpha(r-2)}{r}}$ is an $\frac{\alpha(r-2)}{r}$-stable random variable.
\end{theorem}

When compared with \Cref{thm:gro2} and \Cref{thm:gro2'}, $M_r(A_{\mu,n})$ requires an additional centering and a different scaling. 
The appearance of the centering can be seen from the definition of $M_r(A_{\mu,n})$: 
\begin{align}\label{intuition}
\begin{split}
    M_r(A_{\mu,n})=\max_{\bm x\in\R^n:\|\bm x\|_r\leq1}\left[\mu\left(\bm 1^\top\bm x\right)^2+\bm x^\top A_n\bm x\right].
\end{split}
\end{align}
For $\alpha\in(\alpha_*(r),\bar\alpha_{*,2}(r))$, it follows that $\max_{\|\bm x\|_r\leq 1}\bm x^\top A_n\bm x=O_\delta(b_n)$ \whp\ by \Cref{thm:gro2'} and $\max_{\|\bm x\|_r\leq 1}\mu(\bm 1^\T\bm x)^2= n^{2-\frac{2}{r}}\mu$ by H\"older's inequality. 
When $\alpha>\alpha_*(r)$ and $\mu>0$, the inequality $b_n\ll n^{2-\frac{2}{r}}$ 
 ensures that the dominant term in $M_r(A_n)$ is the first term in \eqref{intuition}, necessitating centering. 
Conversely, when $\alpha<\alpha_*(r)$ as in the setting of \Cref{thm:gro2} or $\mu=0$ as in the setting of \Cref{thm:gro2'}, the second term in \eqref{intuition}, which is responsible for all fluctuations, dominates the first term, rendering centering unnecessary. 
However, the reason for a different {\em scaling} in \Cref{thm:gro3} is more subtle and not completely apparent. 
In the light-tailed case analyzed in \cite{dhara2020r},  a scaling factor of $n^{-\frac{r-2}{r}}\sigma^{-1}$ where $\sigma$ is the variance of the matrix entries, is needed to characterize the fluctuations.  However, in the setting of \Cref{thm:gro3}, the variance is infinite and instead, the additional scaling factor takes the form $\left(b_n^{-1}n\right)^{\frac{r}{r-2}}$, reflecting the difference in behavior in the presence of heavy tails. 

The proof of  \Cref{thm:gro3},   
presented   in \Cref{ss:gro3},  is much more subtle 
than \Cref{thm:gro2} and \Cref{thm:gro2'}. 
First,  one needs to employ a different argument  to show  that  $M_r( A^{\mathsf{sm}}_n)$ is asymptotically negligible.  
Here, one relates $M_r( A^{\mathsf{sm}}_n)$ (or more generally the $r \rightarrow p$ norm of $ A^{\mathsf{sm}}_n$) to the spectral norm and leverages bounds on the spectral norm from random matrix theory.  Second, the estimate on $M_r( A^{\mathsf{int}}_n)$ used in \Cref{thm:gro2} also proves inadequate.   Instead,  the proof 
relies on  a detailed sandwich argument for $M_r(A_n -  A^{\mathsf{sm}}_n)$ (see \Cref{s:rtop3}) where  a general  {\em lower} bound is identified 
and shown to asymptotically match the above mentioned upper bound at the scale of fluctuations.   For intuition behind the ansatz for the lower bound, which is   inspired by Boyd's nonlinear power method (even though is is not directly applicable in this setting), see \Cref{ss:main-lower}.  
It should be mentioned that all the proofs are carried out in the context of general $r \to p$ norms and then adapted to study $M_r(A_n).$

It is natural to ask what happens 
when $\alpha > \bar\alpha_{*,1}(r)$. 
In this case we expect the fluctuations of $M_r(A_n)$ to be dominated by  small, dense matrix terms. Consequently, the approach of leveraging the sparsity of large terms in heavy-tailed entries, as employed in this work, becomes ineffective.  We believe that the analysis of  this case is an interesting open problem.

Finally, as a corollary of \Cref{thm:gro2}, we obtain asymptotics of the ground state energy of Levy spin glass model, which is a result of independent interest. 
Spin glasses are a type of disordered magnetic system characterized by the presence of both ferromagnetic and antiferromagnetic interactions between the spins  of the atoms in the material.
Mathematically, it is a family of  (random) Gibbs  distributions on $n$-spin configurations, parameterized by the quantity $\beta> 0$, referred to as the inverse temperature, and defined as follows: 
\begin{equation}
\begin{aligned}
    P_\beta(\bm \sigma)\propto \exp\left(- \beta H (\bm \sigma)\right), \quad\forall\,\bm\sigma\in\{+1,-1\}^n,
\end{aligned}
\end{equation}
where the so-called Hamiltonian $H: \{+1, -1\}^n \mapsto \R$ takes the form  $H(\bm\sigma)=-\sum_{i,j\in[n]}h_{ij}\sigma_i\sigma_j$ for some $n \times n$ interaction matrix 
$h = (h_{ij})$. 
The Levy spin glass model, which corresponds to the case when the interaction matrix $h$ has independent centered heavy-tailed entries, was introduced in \cite{cizeau1993mean} to understand surprising experimental results on dilute spin glasses with dipolar interactions. 
The ground states of the spin glass refer 
to the most probable configurations, in the limit as $\beta\rightarrow\infty$ (the zero temperature regime).  In turn, these states correspond to  the maximizers of the ``ground state energy''  given by  $\sum_{i,j\in[n]}h_{ij}\sigma_i\sigma_j$.  Recent work \cite{chen2023some,kim2024fluctuations}  studied  the fluctuations of Levy spin glass  for finite $\beta>0$, but 
did not consider the corresponding result for the limit $\beta\rightarrow\infty$.  This is resolved in the following   corollary, whose proof can be found in \Cref{ss:spin}.

\begin{corollary}\label{cor:ground}
Suppose the sequence $A_n=\left(a_{ij}\right)_{i,j\in[n]}$, $n \in \N$, satisfies \Cref{assume} with index $\alpha\in(0,1)$. Then with $b_n$ as defined in \eqref{d:b}, we have
\begin{align}
\begin{split}
    b_n^{-1}\sup_{\bm x\in\left\{-1,1\right\}^n}\bm x^\T A_n\bm x\xrightarrow{(d)} Y_\alpha,
\end{split}
\end{align}
where $Y_\alpha$ is an $\alpha$-stable random variable.
\end{corollary}

\subsection{The \texorpdfstring{$r\rightarrow p$}{r-to-p} operator norm }\label{ss:rtop-results}
We next turn to our results on high-dimensional asymptotics of the $r\rightarrow p$ operator norm  (\ref{opt2}) for heavy-tailed random matrices.  They  parallel those for the Levy-Grothendieck $\ell_r$ problem. 

\begin{theorem}\label{thm:rtop1}
Fix $1\leq r\leq p\leq\infty$. Suppose the sequence $A_n=(a_{ij})_{i,j\in[n]},$ $n \in \N,$ satisfies \Cref{assume} with  $\alpha\in(0,2)$, and let $b_n$ be as defined in \eqref{d:b}. 
Then we have
\begin{equation}
\lim _{n \rightarrow \infty} \mathbb{P}\left(b_n^{-1}\|A_n\|_{r\rightarrow p} \leq x\right)=\exp \left(-x^{-\alpha}\right).
\end{equation}
\end{theorem}
\Cref{thm:rtop1}, whose proof is postponed to \Cref{s:rtop1}, is the operator norm  analogue of \Cref{thm:gro1}.  Note that the scaling in \Cref{thm:rtop1} is missing the additional factor $2^{\frac{2}{r}-1}$ that arises in \Cref{thm:gro1}.   This is because 
when $r \in [1,2]$ or $1 \leq r\leq p \leq \infty$, $M_r(A_n)$ and $\|A_n\|_{r\rightarrow p}$, respectively, are fully determined by the largest entries of $A_n$ (recall that there are two largest entries due to symmetry). 
Now, the optimal vectors $\bm x^*$ and $\bm y^*$  in the dual formulation \eqref{eq:rtop-dual} of \ref{opt2}, can be chosen separately, with the mass of each of these vectors  concentrated on a single coordinate corresponding to the row and column of $A_n$, respectively, that contains the largest entry. 
 However,  since this entry is  unlikely to reside on the diagonal (see \Cref{lemma:typicalBehaviour}), these  coordinates are distinct, say $(i,j)$ with $i\neq j$, and it turns out that the mass of the $\ell_r$-normalized maximizing vector $\bm x^*$ of \ref{opt1} needs to be distributed into multiple pairs $i,j$, each with mass $2^{-1/r}$. 
 

\Cref{thm:rtop2} and \Cref{thm:rtop2'} below, which parallel \Cref{thm:gro2} and \Cref{thm:gro2'} respectively, are proved at the end of \Cref{s:large-fluct}. 
\begin{theorem}\label{thm:rtop2}
Fix $1\leq p<r\leq \infty$, denote $\gamma\coloneqq \frac{1}{p}-\frac{1}{r}$ and define $\alpha_*\equiv\alpha_*(\gamma)\coloneqq \frac{2}{1+\gamma}$. Suppose the sequence $A_n=(a_{ij})_{i,j\in[n]}, n \in \N,$ satisfies \Cref{assume} with  $\alpha$ lying in the interval 
\begin{align}\label{rtop2-cond1}
\begin{split}
    0<\alpha<\alpha_*(\gamma).
\end{split}
\end{align}
Then with $b_n$ defined in \eqref{d:b}, we have
\begin{equation}\label{eqn:rBigp}
b_n^{-1}\|A_n\|_{r\rightarrow p}\xrightarrow{(d)} \left(Y_{\alpha\gamma}\right)^{\gamma},
\end{equation}
where $Y_{\alpha\gamma}$ is an $\alpha\gamma$-stable random variable.
\end{theorem}

\begin{theorem}\label{thm:rtop2'}
Fix $1\leq p<r\leq \infty$, denote $\gamma\coloneqq \frac{1}{p}-\frac{1}{r}$, define $\alpha_*\equiv\alpha_*(\gamma)\coloneqq \frac{2}{1+\gamma}$ and $\bar\alpha_{*,1}\equiv\bar\alpha_{*,1}(\gamma)\coloneqq\min\{2,\frac{1}{\gamma}\}$. Suppose the sequence $A_n=\left(a_{ij}\right)_{i,j\in[n]}$, $n \in \N$, satisfies \Cref{assume} with $\alpha$, such that
\begin{align}\label{rtop2-cond2}
\begin{split}
    (a_{ij})_{i,j\in[n]}\text{ are centered and}\quad \alpha_*(\gamma)\leq \alpha<\bar\alpha_{*,1}(\gamma).
\end{split}
\end{align}
Then with $b_n$ defined in \eqref{d:b}, we have
\begin{equation}\label{eqn:rBigp}
b_n^{-1}\|A_n\|_{r\rightarrow p}\xrightarrow{(d)} \left(Y_{\alpha\gamma}\right)^{\gamma},
\end{equation}
where $Y_{\alpha\gamma}$ is an $\alpha\gamma$-stable random variable.
\end{theorem}

\Cref{thm:rtop3}, whose proof is given in  \Cref{s:rtop3}, is the analogue of \Cref{thm:gro3}: as before, in this regime the non-centered matrix exhibits different asymptotics. 

\begin{theorem}\label{thm:rtop3}
Fix $1< p<r<\infty$, $\mu\neq 0$ and denote $\gamma\coloneqq\frac{1}{p}-\frac{1}{r}$ and $\gamma'\coloneqq \frac{1}{\min\{2,p\}}-\frac{1}{\max\{2,r\}}$. Define $\alpha_*\equiv\alpha_*(\gamma)\coloneqq\frac{2}{1+\gamma}$ and $\bar\alpha_{*,2}\equiv\bar\alpha_{*,2}(\gamma,\gamma')\coloneqq \frac{2-\gamma }{\gamma(\gamma'-\gamma)+1}$. Suppose the sequence $A_n=\left(a_{ij}\right)_{i,j\in[n]}$, $n \in \N$, satisfies \Cref{assume} with $\alpha$, such that
\begin{align}\label{rtop3-cond}
\begin{split}
    \alpha_*(\gamma)<\alpha<\bar\alpha_{*,2}(\gamma,\gamma').
\end{split}
\end{align}
Define
\begin{align}
\begin{split}
    A_{\mu,n}\coloneqq \mu\bm 1\bm 1^\T+A_n=(\mu+a_{ij})_{i,j\in[n]}.
\end{split}
\end{align}
Then with $b_n$ as defined in \eqref{d:b},
\begin{align}
\begin{split}
    b_n^{-\frac{1}{\gamma}}n^{\frac{1}{\gamma}-\gamma}\left(\|A_{\mu,n}\|_{r\rightarrow p}-n^{1+\gamma}|\mu|\right)\rightarrow Y_{\alpha\gamma},
\end{split}
\end{align}
where $Y_{\alpha\gamma}$ is an $\alpha\gamma$-stable random variable.
\end{theorem}

\subsection{Organization of the rest of the paper}\label{ss:organization}
As shown in \eqref{ineq:dual}, the $r\rightarrow r^*$ operator norm provides a natural upper bound for $M_r(A)$. For this reason, we first present the proofs of the $r\rightarrow p$ operator norm, with  the proofs  of \Cref{thm:rtop1}, \Cref{thm:rtop2} and \Cref{thm:rtop3} provided in \Cref{s:rtop1}, \Cref{s:rtop2} and \Cref{s:rtop3}.  The proofs of results on  the $\ell_r$-Grothendieck problem are then completed in  
\Cref{s:gro}. 

\section{Proof of Theorem \ref{thm:rtop1}}\label{s:rtop1}

 For brevity, in this and subsequent sections, we  often suppress the subscript $n$ of $A_n$ when it is clear from the context.    We start by recalling the limiting distribution of the largest entry of a heavy-tailed matrix  (after proper normalization) identified in \cite[(24)]{Sosh04}.
 
\begin{lemma}{\cite[(24)]{Sosh04}}\label{Frechet}
Fix $\alpha\in(0,2)$. Suppose the sequence $A_n=\left(a_{ij}\right)_{i,j\in[n]}$, $n \in \N$, satisfies \Cref{assume}  with $\alpha \in (0,2)$, and let $(b_n)_{n\in\N}$ be as in \eqref{d:b}. Then the following holds:
\begin{align}\label{eqn:a_*Dist}
\begin{split}
    \lim_{n\rightarrow\infty}\P\left(b_n^{-1}\max_{i,j\in[n]}\left|a_{ij}\right|\leq x\right)=\exp\left(-x^{-\alpha}\right).
\end{split}
\end{align}
\end{lemma}

Next, we identify the condition that should be satisfied by the maximizing vector of \ref{opt2} when $1<r,p<\infty$. 

\begin{lemma}{\cite[(4.4)]{dhara2020r}} \label{lemma:characterization-max-vec}
Fix $1<r,p<\infty$ and an $n\times n$ matrix $V$. Let $\bm v\equiv\bm v_n=(v_i)_{i\in[n]}$ be a maximizing vector of \ref{opt2} associated with $V$ with $\|\bm v\|_r=1$.
Then the following relation holds:
\begin{equation}\label{eqn:gammaMaxVec}
V^\trans \Psi_{p}(V \bm{v})=\|V\|_{r\rightarrow p}^p \Psi_r(\bm{v}).
\end{equation}
\end{lemma}

We will make repeated use of a well known result that describes the typical structure of heavy-tailed random matrix $A$. 
\begin{lemma}{\cite[Lemma~4]{Sosh04}} \label{lemma:typicalBehaviour} 
Suppose $A_n=\left(a_{ij}^n\right)_{i,j\in[n]}, n \in \N$,  satisfies \Cref{assume} with $\alpha\in(0,2)$. Then the following statements are true: 
\begin{enumerate}[label=(\alph*)]
\item \whp, 
$\left\{i\in[n]:\ a_{ii}\geq b_n^{11/20}\right\}=\emptyset$;
\item $\forall\,\delta>0$, \whp, $\forall\,i\in[n]$, $\#\left\{j\in[n]:\,\left|a_{ij}\right|\geq b_n^{3/4+\delta}\right\}\leq 1$;
\item $\exists\,\delta(\alpha)>0$, such that, \whp, $\forall\,i\in[n]$, $\sum_{j\in[n]}\left|a_{ij}\right|\bbm 1\left\{\left|a_{ij}\right|\leq b_n^{3/4+\delta(\alpha)}\right\}\leq b_n^{3/4+\frac{\alpha}{8}}$.
\end{enumerate}
\end{lemma}
The following inequality is a direct consequence of Lemma \ref{lemma:typicalBehaviour}.
\begin{corollary}
Suppose $A_n=\left(a_{ij}\right)_{i,j\in[n]}, n \in \N,$ satisfies \Cref{assume} with $\alpha\in(0,2)$. Let $a_*\coloneqq\max_{i,j\in[n]}|a_{ij}|$. Then \whp, we have
\begin{equation}\label{eqn:sumOfTypicalRow}
\max_{i\in[n]}\sum_{j\in[n]} |a_{i j}| \leq a_{*}(1+o(1))\quad\text{and}\quad\max_{j\in[n]}\sum_{i\in[n]}\left|a_{ij}\right|\leq a_*\left(1+o(1)\right).
\end{equation}
\end{corollary}
\begin{proof}
Due to the symmetry of $A$, it suffices to only prove the first inequality. 

Note that $a_*$ is of order $O(b_n)$ according to \eqref{eqn:a_*Dist}. Fix $i\in[n]$ and consider the event $\mathcal E_i$ that $a_*$ appears in row $i$ or there exists one entry in row $i$ that is at least $b_n^{3/4+\delta(\alpha)}$. Lemma \ref{lemma:typicalBehaviour} (b) implies that, \whp\ on $\mathcal E_i$, all the remaining entries are less than $b_n^{3/4+\delta(\alpha)}$. Then by \Cref{lemma:typicalBehaviour} (c), \whp\ on $\mathcal E_i$, the sum of the remaining terms is of order $o(a_*)$.

On the other hand, on the event $\mathcal E_i^c$, that is when $a_*$ is not in row $i$ and every entry in row $i$ is less than $b_n^{3/4+\delta(\alpha)}$, Lemma \ref{lemma:typicalBehaviour} (c) implies that, \whp\ the sum of this row is $o(a_*)$. This proves the corollary.
\end{proof}

\begin{proof}[Proof of Theorem \ref{thm:rtop1}] Let $\bm v=(v_i)_{i\in[n]}$ be a maximizing vector of \ref{opt2} with $\|\bm v\|_r=1$.
Let $v_*\coloneqq\max_{i\in[n]}|v_i|$ and $a_*\coloneqq\max_{i,j\in[n]}|a_{ij}|$. 

We first show that, for $1\leq r\leq p\leq\infty$,
\begin{align}\label{eqn:rEpUpper}
\begin{split}
    \|A\|_{r\rightarrow p}\leq a_*\left(1+o(1)\right).
\end{split}
\end{align}
This is proved separately for three parameter regimes.
\begin{enumerate}
\item Suppose $1<r\leq p<\infty$. Using \Cref{lemma:characterization-max-vec} and the definition of $\Psi_p,\Psi_r$ in \eqref{def:psi} in the first line, the definition of $v_*$ and \eqref{eqn:sumOfTypicalRow} in the second and third lines, respectively, and again using \eqref{eqn:sumOfTypicalRow} in the last line below, we conclude that \whp, for all $i\in[n]$,
\begin{equation}\label{eqn:rEpUpperDerivation}
\begin{aligned}
\|A\|_{r\rightarrow p}^{p} |v_{i}|^{r-1} & = \left|\sum_{j=1}^{n} a_{i j}\left|\sum_{k=1}^{n} a_{j k}v_k\right|^{p-1}\operatorname{sgn}\left(\sum_{k=1}^{n} a_{j k}v_k\right)\right|\\
&\leq\sum_{j=1}^{n} |a_{i j}|\left(\sum_{k=1}^{n} |a_{j k}|\right)^{p-1} v_*^{p-1} \\
& \leq\sum_{j=1}^{n} |a_{i j}|\ a_{*}^{p-1}(1+o(1))v_{*}^{p-1} \\
& \leq a_{*}^{p} v_{*}^{p-1}(1+o(1)).
\end{aligned}
\end{equation}
Choosing $i$ to be the (random) index such that $|v_i|=v_*\leq 1$ in \eqref{eqn:rEpUpperDerivation}, and recalling $r\leq p$,
we have \whp,
\begin{align}
\begin{split}
    \|A\|_{r\rightarrow p}\leq a_*(1+o(1)).
\end{split}
\end{align}
\item Suppose $r=1$. We have
\begin{align}\label{compare11}
\begin{split}
    \|A\|_{1\rightarrow p}=\sup_{\|\bm x\|_1,\|\bm y\|_{p^*}\leq1}\bm x^\T A\bm y\leq \sup_{\|\bm x\|_1,\|\bm y\|_{\infty}\leq 1}\bm x^\T A\bm y= \|A\|_{1\rightarrow1}=\sup_{\|\bm x\|_1=1}\sum_{i\in[n]}\left|\left(A\bm x\right)_i\right|.
\end{split}
\end{align}
Since $\bm x\mapsto \sum_{i\in[n]}\left|\left(A\bm x\right)_i\right|$ is convex, the maximal value must be obtained at extremal points of the unit $\ell_1$-sphere, i.e. $\bm x\in\left\{\bm e_j\right\}_{j\in[n]}$. Hence, \eqref{eqn:sumOfTypicalRow} and \eqref{compare11} imply
\begin{align}
\begin{split}
    \|A\|_{1\rightarrow p}\leq \max_{j\in[n]}\sum_{i\in[n]}\left|a_{ij}\right|\leq a_*\left(1+o(1)\right).
\end{split}
\end{align}
\item Suppose $p=\infty$. We have
\begin{align}\label{compareinfinf}
\begin{split}
    \|A\|_{r\rightarrow\infty}=\sup_{\|\bm x\|_r,\|\bm y\|_{1}\leq1}\bm x^\T A\bm y\leq \sup_{\|\bm x\|_\infty,\|\bm y\|_{1}\leq 1}\bm x^\T A\bm y= \|A\|_{\infty\rightarrow \infty}=\sup_{\|\bm x\|_\infty=1}\max_{i\in[n]}\left|\left(A\bm x\right)_i\right|.
\end{split}
\end{align}
Since $\bm x\mapsto\max_{i\in[n]}$ is convex, the maximal value must be obtained at extremal points of the unit $\ell_\infty$-sphere, i.e. $\bm x\in\left\{1,-1\right\}^n$. Hence \eqref{eqn:sumOfTypicalRow} and \eqref{compareinfinf} imply
\begin{align}
\begin{split}
    \|A\|_{r\rightarrow \infty}\leq \max_{i\in[n]}\sum_{j\in[n]}\left|a_{ij}\right|\leq a_*\left(1+o(1)\right).
\end{split}
\end{align}
\end{enumerate}

On the other hand, let $(i_*,j_*)$ be a (random) pair of indices such that $a_*=a_{i_*j_*}$. Then we have
\begin{equation}\label{eqn:rEpLower}
\|A\|_{r\rightarrow p} \geq\|A \bm e_{i_*}\|_{p} \geq a_{*}.
\end{equation}

Combining \eqref{eqn:rEpUpper} and \eqref{eqn:rEpLower}, we conclude that 
\begin{align}\label{casepp}
\begin{split}
    \|A\|_{r\rightarrow p}=a_*(1+o(1)).
\end{split}
\end{align}
Hence, \Cref{Frechet} implies that $b_n^{-1}\|A\|_{r\rightarrow p}$ converges to the desired distribution.
\end{proof}

\section{Proofs of Theorem \ref{thm:rtop2} and Theorem \ref{thm:rtop2'}}\label{s:rtop2}
Fix $\alpha\in(0,\min\{2,\frac{rp}{r-p}\})$ and $1\leq p<r\leq \infty$. Define
\begin{align}\label{def:eta-para}
\begin{split}
    \eta&\coloneqq\frac{1}{2}\left(1-\frac{\alpha(r-p)}{rp}\right) ,\\
    \zeta&\coloneqq\min\left\{\frac{1}{4}\left(1-\frac{\alpha(r-p)}{rp}\right),1-\frac{\alpha}{2}\right\}.
\end{split}
\end{align}
Under the assumptions of \Cref{thm:rtop2} or \Cref{thm:rtop2'}, we have $0<\zeta<\eta$.
As mentioned earlier, the main strategy of the proof is to
carefully decompose the heavy-tailed matrix $A$ based on the absolute values of entries, and to control the effect of each of these separately. To this end, define
\begin{align}\label{a-split}
\begin{split}
     &A^{\mathsf{sm}}\equiv( a^{\mathsf{sm}}_{ij})\coloneqq \left(a_{ij}\bbm 1\left( |a_{ij}|\leq n^{\frac{1+\eta}{\alpha}}\right)\right),\\
     &A^{\mathsf{la}}\equiv\left( a^{\mathsf{la}}_{ij}\right)\coloneqq \left(a_{ij}\bbm 1\left(|a_{ij}|>n^{\frac{2-\zeta}{\alpha}}\right)\right),\\
     &A^{\mathsf{int}}\equiv ( a^{\mathsf{int}}_{ij})\coloneqq A- A^{\mathsf{sm}}- A^{\mathsf{la}}.
\end{split}
\end{align}

In \Cref{s:small-bound} and \Cref{s:middle-bound}, we show that both $\| A^{\mathsf{sm}}\|_{r\rightarrow p}$ and $\| A^{\mathsf{int}}\|_{r\rightarrow p}$ are of order $o(b_n)$ \whp, where $b_n$ is the constant defined in \eqref{d:b}. Then in \Cref{s:large-fluct}, we analyze the fluctuations of $b_n^{-1}\| A^{\mathsf{la}}\|_{r\rightarrow p}$. 

\subsection{Estimating the small terms}\label{s:small-bound}
The goal of this section is to prove the following estimate.
\begin{proposition}\label{lem:omit-small}
Under the assumptions of \Cref{thm:rtop2} or \Cref{thm:rtop2'}, and with $b_n$ as defined in \eqref{d:b}, \whp,
\begin{align}
\begin{split}
    \| A^{\mathsf{sm}}\|_{r\rightarrow p}=o\left(b_n\right).
\end{split}
\end{align}
\end{proposition}
To prove the result, we start by establishing a general upper bound for the $r\rightarrow p$ norm of an arbitrary $n\times n$ matrix.
\begin{lemma}\label{lem:new-upper-bound}
Fix $1\leq p<r\leq \infty$. For any symmetric matrix $V=(v_{ij})_{i,j\in[n]}\in\mathbb R^n$, we have
\begin{align}\label{eqn:new-upper-bound}
\begin{split}
    \left\|V\right\|_{r\rightarrow p}\leq \left(\sum_{i\in[n]}\left(\sum_{j\in[n]}|v_{ij}|\right)^{\frac{rp}{r-p}}\right)^{\frac{r-p}{rp}}.
\end{split}
\end{align}
\end{lemma}
\begin{proof}
Fix any $\bm x,\bm y\in\R^n$ such that $\|\bm x\|_r=\|\bm y\|_{p^*}=1$. Since $r>p$, applying H\"older's inequality twice we obtain
\begin{align}\label{less-instance}
\begin{split}
    \sum_{i,j\in[n]}v_{ij}x_iy_j&\leq \sum_{i,j\in[n]}|v_{ij}|^{\frac{p}{rp-r+p}}|x_i|\cdot |v_{ij}|^{\frac{rp-r}{rp-r+p}}|y_j|\\
    &\leq \left(\sum_{i,j\in[n]}|v_{ij}||x_i|^{\frac{rp-r+p}{p}}\right)^{\frac{p}{rp-r+p}}\left(\sum_{i,j\in[n]}|v_{ij}||y_j|^{\frac{rp-r+p}{rp-r}}\right)^{\frac{rp-r}{rp-r+p}}\\
    &\leq \left(\sum_{i\in[n]}|x_i|^r\right)^{\frac{1}{r}}\left(\sum_{i\in[n]}\left(\sum_{j\in[n]}|v_{ij}|\right)^{\frac{rp}{r-p}}\right)^{\frac{r-p}{r(rp-r+p)}}\\
    &\quad \times \left(\sum_{j\in[n]}|y_j|^{p^*}\right)^{\frac{1}{p^*}}\left(\sum_{j\in[n]}\left(\sum_{i\in[n]}|v_{ij}|\right)^{\frac{rp}{r-p}}\right)^{\frac{r-p}{p^*(rp-r+p)}}\\
    &= \left(\sum_{i\in[n]}\left(\sum_{j\in[n]}|v_{ij}|\right)^{\frac{rp}{r-p}}\right)^{\frac{r-p}{rp}},
\end{split}
\end{align}
where we use $\|\bm x\|_r=\|\bm y\|_{p^*}=1$ and the symmetry of $V$ for the last equality. In view of the corresponding dual formulation \eqref{eq:rtop-dual} of $\|V\|_{r\rightarrow p}$,
one can take the supremum over $\bm x,\bm y$ with $\|\bm x\|_r=\|\bm y\|_{p^*}=1$ on the left-hand side of \eqref{less-instance} to complete the proof.
\end{proof}
We now bound the row sums of $ A^{\mathsf{sm}}$. Note that in conjunction with \Cref{lem:new-upper-bound}, this will yield an upper bound for $\| A^{\mathsf{sm}}\|_{r\rightarrow p}$.
To this end, we need the following standard concentration result \cite[Theorem~2.8.4]{Versh18}.  
\begin{lemma}[Bernstein's inequality]\label{bers-ineq}
Let $X_1, \ldots, X_n$ be independent, mean zero random variables, such that $\left|X_k\right| \leq K$ for all $k\in[n]$. Then, for every $t \geq 0$, we have
\begin{equation}
\mathbb{P}\left\{\left|\sum_{k=1}^n X_k\right| \geq t\right\} \leq 2 \exp \left(-\frac{t^2 / 2}{\sigma^2+K t / 3}\right),
\end{equation}
where $\sigma^2=\sum_{k=1}^n \mathbb{E} [X_k^2]$ is the variance of the sum.
\end{lemma}
\begin{lemma}\label{lemma:remainingSmall}
Under the assumptions of \Cref{thm:rtop2} or \Cref{thm:rtop2'} and with $ A^{\mathsf{sm}}$ defined as in \eqref{a-split},   \whp for all $i\in[n]$ we have 
\begin{align}
\begin{split}
    \left|\sum_{i=1}^n a^{\mathsf{sm}}_{ij}\right|\leq \sum_{j=1}^n| a^{\mathsf{sm}}_{ij}|=O_\delta\left(n^{\max\left\{\frac{1+\eta}{\alpha},1\right\}}\right).
\end{split}
\end{align}
\end{lemma}
\begin{proof}
First, we calculate the second moment of $ a^{\mathsf{sm}}_{ij}$ using $\bar F$ from \eqref{ccdf} and \eqref{eqn:slowly-varying-property} to obtain
\begin{align}
\begin{split}
    \E\left[ (a^{\mathsf{sm}}_{ij})^2\right]\leq \int_{0}^{n^{2(1+\eta)/\alpha}}\left(\bar F\left(x^{\frac{1}{2}}\right)-\bar F\left(n^{\frac{1+\eta}{\alpha}}\right)\right)dx=O_\delta\left(n^{\left(\frac{2}{\alpha}-1\right)\left(1+\eta\right)}\right).
\end{split}
\end{align}
This implies that
\begin{align}\label{var-calc}
\begin{split}
    \operatorname{Var}\left(\sum_{j=1}^n| a^{\mathsf{sm}}_{ij}|\right)\leq n\E\left[ (a^{\mathsf{sm}}_{11})^2\right]=O_\delta\left(n^{\frac{2}{\alpha}+\eta\left(\frac{2}{\alpha}-1\right)}\right).
\end{split}
\end{align}
We now consider three cases. First, suppose $\alpha<1$. Then
\begin{align}\label{case1}
\begin{split}
    \E\left[\sum_{j=1}^n| a^{\mathsf{sm}}_{ij}|\right]=n\int_{0}^{n^{(1+\eta)/\alpha}}\left(\bar F(x)-\bar F\left(n^{\frac{1+\eta}{\alpha}}\right)\right)dx=O_\delta\left(n^{\frac{1}{\alpha}+\eta\left(\frac{1}{\alpha}-1\right)}\right),
\end{split}
\end{align}
where we again used \eqref{eqn:slowly-varying-property} in the last equality.
Next suppose $\alpha=1$. Then
similarly, we have
\begin{align}\label{case2}
\begin{split}
    \E\left[\sum_{j=1}^n| a^{\mathsf{sm}}_{ij}|\right]=n\int_{0}^{n^{(1+\eta)/\alpha}}\left(\bar F(x)-\bar F\left(n^{\frac{1+\eta}{\alpha}}\right)\right)dx=O_\delta(n\log n)=O_\delta\left(n\right).
\end{split}
\end{align}
Finally, in the case $\alpha\in\left(1,\min\{2,\frac{rp}{r-p}\}\right)$, we have  $\E\left[\left|a_{11}\right|\right]<\infty$, which implies 
\begin{align}\label{case3}
\begin{split}
    \E\left[\sum_{j=1}^n| a^{\mathsf{sm}}_{ij}|\right]\leq n\E\left[\left|a_{11}\right|\right] =O_\delta\left(n\right).
\end{split}
\end{align}
Since $\left| a^{\mathsf{sm}}_{ij}\right|\leq n^{\frac{1+\eta}{\alpha}}$ by \eqref{a-split}, \Cref{bers-ineq} and \eqref{var-calc} together imply that for any $\delta>0$, 
\begin{align}\label{bern}
\begin{split}
    &\P\left(\left|\sum_{j=1}^n| a^{\mathsf{sm}}_{ij}|-\E\left[\sum_{j=1}^n| a^{\mathsf{sm}}_{ij}|\right]\right|\geq n^{\frac{1+\eta}{\alpha}+\delta}\right)\\
    &\leq 2\exp\left(-\frac{cn^{\frac{2(1+\eta)}{\alpha}+2\delta}}{\Var\left(\sum_{j=1}^n| a^{\mathsf{sm}}_{ij}|\right)+n^{\frac{1+\eta}{\alpha}}n^{\frac{1+\eta}{\alpha}+\delta}}\right)\\
    &=2\exp\left(-\frac{cn^\delta}{O_\delta\left(n^{-\eta}\right)+1}\right)\\
    &\leq 2\exp\left(-cn^{\delta}\right).
\end{split}
\end{align}
To complete the proof, combine the last four displays with the union bound.
\end{proof}
We now establish three additional lemmas that hold under the assumptions of \Cref{thm:rtop2'}. The first one bounds the mean of $ a^{\mathsf{sm}}_{ij}$ as defined in \eqref{a-split}, the second lemma concerns the spectral norm of $ A^{\mathsf{sm}}$ and the third one relates the spectral norm to the $r\rightarrow p$ operator norm.
\begin{lemma}\label{l:near-centered}
Fix $1\leq p<r\leq \infty$ and suppose \Cref{assume} holds with index $\alpha$ and \eqref{rtop2-cond2}.
Then, for $i,j\in[n]$, the matrix $ A^{\mathsf{sm}}$ defined in 
\eqref{a-split} satisfies 
\begin{align}
\begin{split}
    \left|\mathbb E\left[ a^{\mathsf{sm}}_{ij}\right]\right|=O\left(n^{\frac{1}{\alpha}-1}\right).
\end{split}
\end{align}
\end{lemma}
\begin{proof} Fix $i,j\in[n]$.
By the mean zero condition on $a_{ij}$, we have
\begin{align}
\begin{split}
    \int_0^\infty \P\left(a_{ij}>x\right)dx = \int_0^\infty \P\left(a_{ij}<-x\right)dx\eqqcolon K_0.
\end{split}
\end{align}
By the definition of $ a^{\mathsf{sm}}_{ij}$ in \eqref{a-split}, it follows that
\begin{align}
\begin{split}
    \E\left[ a^{\mathsf{sm}}_{ij}\right]=&\int_{0}^{n^{\frac{1+\eta}{\alpha}}}\P\left(x<a_{ij}\leq n^{\frac{1+\eta}{\alpha}}\right)dx - \int_{0}^{n^{\frac{1+\eta}{\alpha}}}\P\left(-n^{\frac{1+\eta}{\alpha}}\leq a_{ij}<-x\right)dx\\
    =&\left(K_0-\int_{n^{\frac{1+\eta}{\alpha}}}^\infty\P\left(a_{ij}>x\right)dx-n^{\frac{1+\eta}{\alpha}}\P\left(a_{ij}>n^{\frac{1+\eta}{\alpha}}\right)\right)-\\
    &\left(K_0-\int_{n^{\frac{1+\eta}{\alpha}}}^\infty\P\left(a_{ij}<-x\right)dx-n^{\frac{1+\eta}{\alpha}}\P\left(a_{ij}<-n^{\frac{1+\eta}{\alpha}}\right)\right)\\
    =& \int_{n^{\frac{1+\eta}{\alpha}}}^\infty\P\left(a_{ij}<-x\right)dx-\int_{n^{\frac{1+\eta}{\alpha}}}^\infty\P\left(a_{ij}>x\right)dx\\
    &+n^{\frac{1+\eta}{\alpha}}\P\left(a_{ij}<-n^{\frac{1+\eta}{\alpha}}\right)-n^{\frac{1+\eta}{\alpha}}\P\left(a_{ij}>n^{\frac{1+\eta}{\alpha}}\right).
\end{split}
\end{align}
This implies that
\begin{align}
\begin{split}
    \left|\E\left[ a^{\mathsf{sm}}_{ij}\right] \right|\leq \int_{n^{\frac{1+\eta}{\alpha}}}^\infty \P\left(|a_{ij}|>x\right)dx + n^{\frac{1+\eta}{\alpha}}\P\left(|a_{ij}|>n^{\frac{1+\eta}{\alpha}}\right)=O_\delta\left(n^{\frac{1}{\alpha}-1+\frac{\eta(1-\alpha)}{\alpha}}\right)=O\left(n^{\frac{1}{\alpha}-1}\right),
\end{split}
\end{align}
where we used \eqref{ccdf} in the first equality and $\alpha>1$ from \eqref{rtop2-cond2} in the last equality.
\end{proof}
\begin{lemma}\label{l:path-counting}
Fix $1\leq p<r\leq\infty$ and suppose \Cref{assume} is satisfied  with index $\alpha$ and \eqref{rtop2-cond2} also holds. Let $b_n$ and $ A^{\mathsf{sm}}$ be as defined in \eqref{d:b} and \eqref{a-split}, respectively. Then \whp,
\begin{align}\label{goal2}
\begin{split}
    \| A^{\mathsf{sm}}\|_{2\rightarrow 2}=O_\delta(n^{\frac{1+\eta}{\alpha}}).
\end{split}
\end{align}
\end{lemma}
\begin{proof}
Let $\mu\coloneqq\E\left[ a^{\mathsf{sm}}_{11}\right]$. Using \Cref{l:near-centered} in the last step, we have
\begin{align}
\begin{split}
    \left|\| A^{\mathsf{sm}}\|_{2\rightarrow 2}-\| A^{\mathsf{sm}}-\mu\bm 1\bm 1^\T\|_{2\rightarrow 2}\right|\leq \|\mu\bm 1\bm 1^\T\|_{2\rightarrow 2}=\mu n=O\left(n^{\frac{1}{\alpha}}\right).
\end{split}
\end{align}
Hence, to show \eqref{goal2}, we may assume without loss of generality that $\{ a^{\mathsf{sm}}_{ij}:\,i,j\in[n]\}$ are centered. 
Furthermore, by a standard symmetrization argument (see \cite[(2.60)]{tao2023topics}), in the rest of the proof we may assume $ a^{\mathsf{sm}}_{ij}\stackrel{d}{=}- a^{\mathsf{sm}}_{ij}$.

We first show that
\begin{align}\label{fir-claim}
\begin{split}
    \E\left[\| A^{\mathsf{sm}}\|_{2\rightarrow 2}\right]=O_\delta\left(n^{\frac{1+\eta}{\alpha}}\right).
\end{split}
\end{align}
To prove \eqref{fir-claim}, define $\sigma_{ij}\coloneqq\left(\E\left[ (a^{\mathsf{sm}}_{ij})^2\right]\right)^{1/2}$ and $\xi_{ij}\coloneqq\left(\E\left[ (a^{\mathsf{sm}}_{ij})^2\right]\right)^{-1/2} a^{\mathsf{sm}}_{ij}$, and $\sigma\coloneqq \max_{i\in[n]}\left(\sum_{j\in[n]}\sigma_{ij}^2\right)^{1/2}$. Then it is easy to verify that
\begin{align}
\begin{split}
    \sigma=\left(n\E\left[ (a^{\mathsf{sm}}_{11})^2\right]\right)^{\frac{1}{2}}&=O_\delta\left(n^{\frac{1+\eta}{\alpha}}\right),\\
    \max_{i,j}\E\left[\left|\sigma_{ij}\xi_{ij}\right|^{2\lceil 3\log n \rceil}\right]^{\frac{1}{2\lceil 3\log n \rceil}}&=O_\delta\left(n^{\frac{1+\eta}{\alpha}}\right).
\end{split}
\end{align}
Then \cite[Corollary~3.6]{bandeira2016sharp}, together with the assumption $ a^{\mathsf{sm}}_{ij}\stackrel{d}{=}- a^{\mathsf{sm}}_{ij}$, yields the bound
\begin{align}
\begin{split}
    \E\left[\| A^{\mathsf{sm}}\|_{2\rightarrow 2}\right]\leq e^{2/3}\left[2\sigma+52\max_{i,j}\E\left[\left|\sigma_{ij}\xi_{ij}\right|^{2\lceil 3\log n \rceil}\right]^{\frac{1}{2\lceil 3\log n \rceil}}\sqrt{\log n}\right]=O_{\delta}\left(n^{\frac{1+\eta}{\alpha}}\right).
\end{split}
\end{align}
This proves the claim \eqref{fir-claim}.

Next, define the map $F:\mathbb R^{(n+1)n/2}\rightarrow \mathbb R$ as follows: given $\bm x=\left(x_{ij}\right)_{1\leq i\leq j\leq n}$, $F(\bm x)$ is the $2\rightarrow 2$ norm of the $n\times n$ symmetric matrix with upper triangular part equal to $\left(x_{ij}\right)_{1\leq i\leq j\leq n}$. Since $F$ is convex with Lipschitz constant $1$ and $| a^{\mathsf{sm}}_{ij}|\leq n^{\frac{1+\eta}{\alpha}}$ by \eqref{a-split}, Talagrand's concentration inequality \cite[Theorem~2.1.13]{tao2023topics} implies that
\begin{align}
\begin{split}
    \mathbb{P}\left(\| A^{\mathsf{sm}}\|_{2\rightarrow 2}-\mathbb{E}\| A^{\mathsf{sm}}\|_{2\rightarrow 2} \geq \lambda n^{\frac{1+\eta}{\alpha}}\right) \leq C \exp \left(-c \lambda^2\right).
\end{split}
\end{align}
When combined with \eqref{fir-claim}, this completes the proof of the lemma.
\end{proof}
\begin{lemma}\label{l:later-claim}
Fix $1\leq p<r\leq\infty$ and 
suppose \Cref{assume} is satisfied  with index $\alpha$ and \eqref{rtop2-cond2} also holds.  
Let $b_n$ and $ A^{\mathsf{sm}}$ be as defined in \eqref{d:b} and \eqref{a-split}, respectively. Then \whp,
\begin{align}\label{later-claim}
\begin{split}
    \| A^{\mathsf{sm}}\|_{r\rightarrow p}=O_\delta\left(n^{\frac{1}{\min\{2,p\}}-\frac{1}{\max\{2,r\}}+\frac{1+\eta}{\alpha}}\right).
\end{split}
\end{align}
\end{lemma}
\begin{proof}
We prove this by considering three cases:
\begin{enumerate}
\item Suppose $2\leq p<r\leq \infty$. First note that for any $\bm x\in\mathbb R^n$,
\begin{align}\label{fir-note}
\begin{split}
    \|\bm x\|_p =\sup_{\|\bm y\|_{p^*}\leq 1}\left\langle \bm x,\bm y\right\rangle\leq \sup_{\|\bm y\|_2\leq 1}\left\langle \bm x,\bm y\right\rangle=\|\bm x\|_2,
\end{split}
\end{align}
where we use $p^*\leq 2$ in the inequality. Using \eqref{fir-note}, H\"older's inequality and \Cref{l:path-counting}, we have \whp,
\begin{align}\label{case111}
\begin{split}
    \| A^{\mathsf{sm}}\|_{r\rightarrow p}=\sup_{\bm x\neq 0}\frac{\| A^{\mathsf{sm}}\bm x\|_p}{\|\bm x\|_r}\leq \sup_{\bm x\neq 0}\frac{\| A^{\mathsf{sm}}\bm x\|_{2}}{\|\bm x\|_r}\leq \sup_{\bm x\neq 0}\frac{\| A^{\mathsf{sm}}\bm x\|_2}{n^{\frac{1}{r}-\frac{1}{2}}\|\bm x\|_2}=n^{\frac{1}{2}-\frac{1}{r}}\| A^{\mathsf{sm}}\|_{2\rightarrow 2}=O_\delta(n^{\frac{1}{2}-\frac{1}{r}+\frac{1+\eta}{\alpha}}).
\end{split}
\end{align}
\item Suppose $1\leq p<r\leq 2$, notice that $\| A^{\mathsf{sm}}\|_{r\rightarrow p}=\| A^{\mathsf{sm}}\|_{p^*\rightarrow r^*}$ and $2\leq r^*<p^*\leq \infty$. The result in this case follows as in Case 1.
\item Finally, suppose $1\leq p<2<r\leq \infty$. Then H\"older's inequality and \Cref{l:path-counting} imply
\begin{align}
\begin{split}
    \| A^{\mathsf{sm}}\|_{r\rightarrow p}=\sup_{x\neq 0}\frac{\| A^{\mathsf{sm}}x\|_{p}}{\|x\|_r}\leq\sup_{x\neq 0} \frac{n^{\frac{1}{p}-\frac{1}{2}}\| A^{\mathsf{sm}}x\|_2}{n^{\frac{1}{r}-\frac{1}{2}}\|x\|_2}=n^{\frac{1}{p}-\frac{1}{r}}\| A^{\mathsf{sm}}\|_{2\rightarrow 2}=O_\delta(n^{\frac{1}{p}-\frac{1}{r}+\frac{1+\eta}{\alpha}}).
\end{split}
\end{align}
\end{enumerate}
Together, the three cases imply \eqref{later-claim}.
\end{proof}
\begin{proof}[Proof of \Cref{lem:omit-small}]
Suppose $1\leq p<r\leq\infty$ and \eqref{rtop2-cond1} holds.
Invoking first \Cref{lem:new-upper-bound}, then \Cref{lemma:remainingSmall} and finally using \eqref{rtop2-cond1} and \eqref{def:eta-para} for the last equality, we obtain
\begin{align}
\begin{split}
    \| A^{\mathsf{sm}}\|_{r\rightarrow p}\leq \left(\sum_{i\in[n]}\left(\sum_{j\in[n]}| a^{\mathsf{sm}}_{ij}|\right)^{\frac{rp}{r-p}}\right)^{\frac{r-p}{rp}}=O_{\delta}\left(n^{\max\left\{1,\frac{1+\eta }{\alpha}\right\}+\frac{r-p}{rp}}\right)=o(b_n).
\end{split}
\end{align}
This completes the proof under the assumptions of \Cref{thm:rtop2}.

The conclusion under the assumptions of \Cref{thm:rtop2'} follows on combining
\eqref{rtop2-cond2}, \eqref{def:eta-para} and \eqref{later-claim}.
\end{proof}

\subsection{Estimating terms of intermediate magnitude}\label{s:middle-bound}
Note from \eqref{a-split} that
\begin{align}\label{d:a-tilde}
\begin{split}
     a^{\mathsf{int}}_{ij}\coloneqq a_{ij}\bbm 1\left(n^{\frac{1+\eta}{\alpha}}<|a_{ij}|\leq n^{\frac{2-\zeta}{\alpha}}\right).
\end{split}
\end{align}
The goal of this section is to establish the following estimate.
\begin{proposition}\label{lem:omit-middle}
Recall $b_n$ in \eqref{d:b}. Under the assumptions of \Cref{thm:rtop2} or \Cref{thm:rtop2'}, \whp,
\begin{align}
\begin{split}
    \| A^{\mathsf{int}}\|_{r\rightarrow p}=o\left(b_n\right).
\end{split}
\end{align}
\end{proposition}
The proof of this result, which is postponed to the end of this section, relies on two auxiliary results. The first one, \Cref{l:row-sparse}, establishes a  sparsity property of $ A^{\mathsf{int}}$.
\begin{lemma}\label{l:row-sparse}
Under the assumptions of \Cref{thm:rtop2} or \Cref{thm:rtop2'}, there exists $\kappa\equiv \kappa(\eta)\in\N$ independent of $n$ such that, \whp, there exist at most $\kappa$ nonzero entries in any row or column of $ A^{\mathsf{int}}$.
\end{lemma}
\begin{proof}
Define 
\begin{align}\label{d:kappa}
\begin{split}
    \kappa\equiv\kappa(\eta)\coloneqq \left\lceil \frac{1}{\eta}\right\rceil+1.
\end{split}
\end{align}
Fix $n\in\mathbb N$, define the event $E_n$ as follows:
\begin{align}
\begin{split}
    E_n\coloneqq \left\{\exists\,i\in[n]\text{ such that }\exists\,\kappa\text{ nonzero entries on row (or column) }i\text{ of }A\right\}.
\end{split}
\end{align}
By \eqref{ccdf}, \eqref{d:a-tilde} and \eqref{d:kappa}, it follows that
\begin{align}\label{est-kappa}
\begin{split}
    \P\left(E_n\right)\leq  \P\left(|a_{ij}|\geq n^{\frac{1+\eta}{\alpha}}\right)^\kappa \binom{n}{\kappa}n =O_{\delta}\left(n^{1-\kappa\eta}\right)=O_{\delta}(n^{-1}),
\end{split}
\end{align}
which converges to zero as $n\rightarrow\infty$.
This completes the proof of the lemma.
\end{proof}
\begin{remark}
As a simple consequence, with $\kappa$ as in \Cref{l:row-sparse}, for any $s,t> 0$, we have \whp, for all $i\in[n]$, 
\begin{align}
\begin{split}
    \left(\sum_{j\in[n]}| a^{\mathsf{int}}_{ij}|\right)^s\leq\kappa^s \max_{j\in[n]}| a^{\mathsf{int}}_{ij}|^s\leq  \kappa^s \sum_{j\in[n]}| a^{\mathsf{int}}_{ij}|^s. \label{sparse-consequence}
\end{split}
\end{align}
Using the rearrangement inequality \cite[Theorem~368]{hardy1952inequalities}, we also have
\begin{align}
\begin{split}
    \sum_{j\in[n]}\left| a^{\mathsf{int}}_{ij}\right|^s\sum_{k\in[n]}\left| a^{\mathsf{int}}_{ik}\right|^s\leq \kappa \sum_{j\in[n]}\left| a^{\mathsf{int}}_{ij}\right|^{s+t}.\label{sparse-consequence2}
\end{split}
\end{align}
\end{remark}
We now state the second auxiliary result, which exploits the heavy-tailed nature of $\left(a_{ij}\right)_{i,j\in[n]}$.
\begin{lemma}\label{l:sum-small}
Let $\left(X_\ell\right)_{\ell\in[n(n+1)/2]}$ be i.i.d. copies of a nonnegative heavy-tailed random variable with index $\alpha\in(0,2)$. Fix $q>\alpha$ and $\beta<\frac{2}{\alpha}$. Define $\tilde X_\ell\coloneqq X_\ell\bbm 1\left\{X_\ell\leq n^\beta\right\}$. Then \whp,
\begin{align}\label{goalgoal}
\begin{split}
    \sum_{\ell\in[n(n+1)/2]}\tilde X_\ell^q =O_\delta\left(n^{q\beta-\alpha\beta+2}\right).
\end{split}
\end{align}
\end{lemma}
\begin{proof}
We prove \eqref{goalgoal} by directly calculating the second moment of the left-hand side of \eqref{goalgoal}. Indeed, note that by \Cref{d:heavy} and \Cref{l:slowly-varying-property}, it follows that
\begin{align}
\begin{split}
    &\E\left[\tilde X_1^{2q}\right]\leq O_\delta(1)\int_{0}^{n^{2q\beta}}x^{-\frac{\alpha}{2q}}dx=O_\delta\left(n^{2q\beta-\beta\alpha}\right),\\
    &\E\left[\tilde X_1^q\right]\leq O_\delta(1)\int_0^{n^{q\beta}}x^{-\frac{\alpha}{q}}dx=O_\delta\left(n^{q\beta-\beta\alpha}\right).
\end{split}
\end{align}
Therefore, by the assumption $\beta<\frac{2}{\alpha}$, it follows that
\begin{align}
\begin{split}
    \E\left[\left(\sum_{i\in[n(n+1)/2]}X_i^q\right)^2\right]\leq n^2\E\left[\tilde X_1^{2q}\right]+n^4\left(\E\left[\tilde X_1^q\right]\right)^2=O_\delta\left(n^{2q\beta-2\beta\alpha+4}\right).
\end{split}
\end{align}
Together with Markov's inequality, this proves \eqref{goalgoal}.
\end{proof}
\begin{proof}[Proof of \Cref{lem:omit-middle}]
Let $\kappa$ be as in \Cref{l:row-sparse}. By \Cref{lem:new-upper-bound} and \eqref{sparse-consequence}, it follows that, \whp, for each $i\in[n]$,
\begin{equation}
\begin{aligned}
    \| A^{\mathsf{int}}\|_{r\rightarrow p}\leq \left(\sum_{i\in[n]}\left(\sum_{j\in[n]}| a^{\mathsf{int}}_{ij}|\right)^{\frac{rp}{r-p}}\right)^{\frac{r-p}{rp}}\leq \kappa\left(\sum_{i,j\in[n]}| a^{\mathsf{int}}_{ij}|^{\frac{rp}{r-p}}\right)^{\frac{r-p}{rp}}.
\end{aligned}
\end{equation}
By \eqref{d:a-tilde}, we have
\begin{align}\label{seq2}
\begin{split}
    \sum_{i,j\in[n]}| a^{\mathsf{int}}_{ij}|^{\frac{rp}{r-p}}\leq \sum_{i,j\in[n]}|a_{ij}|^{\frac{rp}{r-p}}\bbm 1\left(|a_{ij}|\leq n^{\frac{2-\zeta}{\alpha}}\right)\leq 2\sum_{1\leq i\leq j\leq n}|a_{ij}|^{\frac{rp}{r-p}}\bbm 1\left(|a_{ij}|\leq n^{\frac{2-\zeta}{\alpha}}\right).
\end{split}
\end{align}
Under assumptions of \Cref{thm:rtop2} or \Cref{thm:rtop2'}, we have $\alpha<\frac{rp}{r-p}$. 
Applying \Cref{l:sum-small} with $\left\{X_\ell\right\}_{\ell\in[n(n+1)/2]}=\left\{|a_{ij}|\right\}_{1\leq i\leq j\leq n}$, $q=\frac{rp}{r-p}$ and $\beta=\frac{2-\zeta}{\alpha}$, and with $b_n$ as in \eqref{d:b}, we have \whp,
\begin{align}\label{seq3}
\begin{split}
    \sum_{1\leq i\leq j\leq n}|a_{ij}|^{\frac{rp}{r-p}}\bbm 1\left\{|a_{ij}|\leq n^{\frac{2-\zeta}{\alpha}}\right\}=O_\delta\left(n^{\frac{rp}{r-p}\frac{2}{\alpha}-\zeta\left(\frac{1}{\alpha}\frac{rp}{r-p}-1\right)}\right)=o\left(b_n^{\frac{rp}{r-p}}\right) .
\end{split}
\end{align}
Together, the last three displays prove the proposition.
\end{proof}

\subsection{Characterizing the fluctuation due to  large terms}\label{s:large-fluct}
We now turn our attention to the matrix $ A^{\mathsf{la}}$ in \eqref{a-split}, starting with the following lemma that counts the number of nonzero entries in $ A^{\mathsf{la}}$.
\begin{lemma}\label{lemma:num-large-entry}
Let $\left(X_\ell\right)_{\ell\in[n(n+1)/2]}$ be i.i.d. copies of a heavy-tailed random variable with index $\alpha\in(0,2)$. Fix any $\theta\in(0,1)$, and let $N_n=\#\left\{X_\ell:|X_\ell|> n^{\frac{2\theta}{\alpha}}\right\}$. Then \whp,
\begin{equation}
N_n=\Theta_\delta\left(n^{2-2\theta}\right).
\end{equation}
\end{lemma}
\begin{proof}
Letting $\bar F$ denote complement cumulative distribution function of $X_\ell$, note that $N_n\sim \operatorname{Ber}\left(\frac{n(n+1)}{2},\bar{F}\left(n^{\frac{2\theta}{\alpha}}\right)\right)$. Thus, by \eqref{ccdf}, we have
\begin{equation}\label{eqn:upper-expectation-bound}
\begin{aligned}
\mathbb{E}\left[N_{n}\right] &=\frac{n(n+1)}{2} \cdot \bar{F}\left(n^{\frac{2\theta}{\alpha}}\right) \\
&=\frac{n(n+1)}{2} L\left(n^{\frac{2\theta}{\alpha}}\right)n^{-2\theta}\\
&=\Theta_\delta\left(n^{2-2 \theta}\right),
\end{aligned}
\end{equation}
where the last equality uses \eqref{eqn:slowly-varying-property}.
An application of Chernoff's bound \cite[Theorem~4.4 and Theorem~4.5]{MU17} then implies that
\begin{equation}
\begin{aligned}
\P\left(|N_n-\E[N_n] |>\frac{1}{2}\E[N_n]\right)&\leq2\exp\left(-\Theta_{\delta}\left(n^{2-2\theta}\right)\right).
\end{aligned}
\end{equation}
Together with \eqref{eqn:upper-expectation-bound}, the assumption $\theta\in(0,1)$ and the union bound, this proves the lemma.
\end{proof}
We now state the second lemma, which simplifies the structure of $ A^{\mathsf{la}}$ while preserving its $r\rightarrow p$ norm (and also the optimal value of \ref{opt1} which is going to be used in \Cref{ss:gro2}).
\begin{lemma}\label{modification}
Let $V$ be a symmetric $n\times n$ matrix that satisfies the following three properties:
\begin{enumerate}[(1),topsep=0pt,partopsep=0pt]
\item zero diagonal entries, that is, $V_{ii}=0,\,\forall i\in[n]$;
\item there is at most one non-zero entry in each row and each column;
\item $\exists\,m\in\N$ such that $m<\frac{n}{2}$ and $\#\left\{(i,j)\in[n]\times [n]:\,V_{ij}\neq 0\right\}=2m$.
\end{enumerate}
Then there exists a symmetric $n\times n$ matrix $\check V$ with $\left\{\check V_{ij}:\,i,j\in[n]\right\}=\left\{V_{ij}:\, i,j\in[n]\right\}$ that satisfies properties (1) $\sim$ (3) above and in addition, also satisfies three additional properties:
\begin{enumerate}[(1),start=4]
\item $\check V_{ij}=0$ if $\max\left\{i,j\right\} > 2m$;
\item For $k\in[m]$, $\check V_{2k-1,2k}=\check V_{2k,2k-1}$ has a magnitude that is $k$-th largest in $\left\{|V_{ij}|:\, 1\leq i<j\leq n\right\}$;
\item $\|\check V\|_{r\rightarrow p}=\|V\|_{r\rightarrow p}$ and $M_r(\check V)=M_r(V)$, where we recall that $M_r(V),M_r(\check V)$ are the optimal values of the corresponding $\ell_r$-Grothendieck problems defined in \ref{opt1}.
\end{enumerate}
Moreover, one can ensure that the map from $V$ to $\check V$ is measurable. 
\end{lemma}
\begin{remark}
The output matrix $\check V$ of \Cref{modification} is illustrated in \Cref{fig:config-1}, where blue region represents region of zero entries and the red squares represent nonzero entries. Note that each row or column of $\check V$ has at most one red square and all red squares are off-diagonal.
\end{remark}
\begin{figure}[ht!]
\centering
\includegraphics[width=0.31\textwidth]{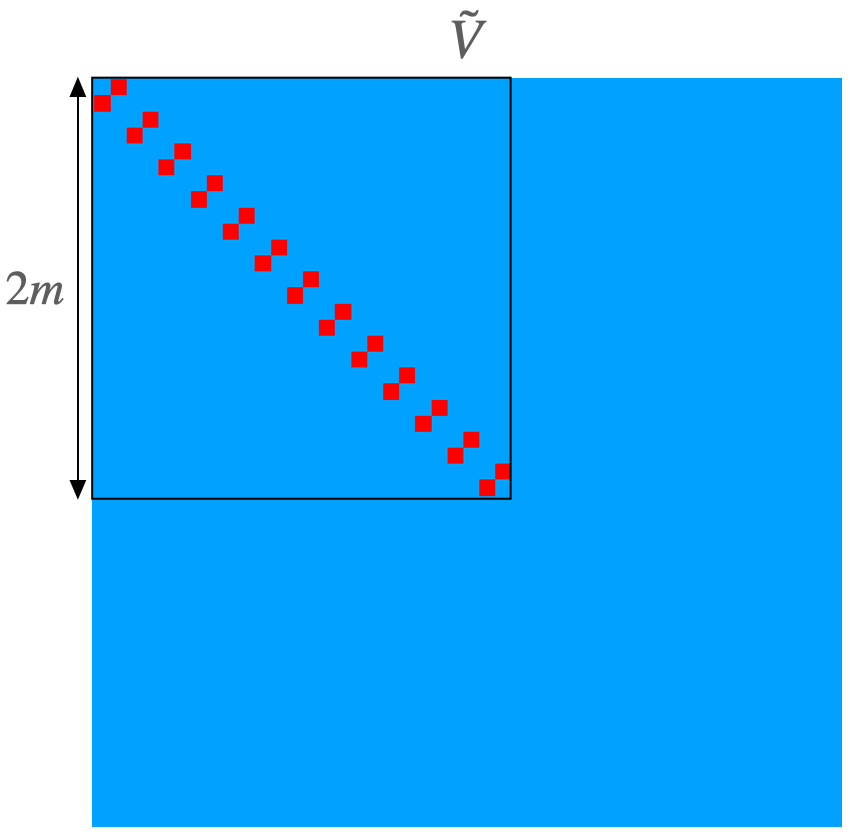}
\centering
\caption{Structure of $\check V$}
\label{fig:config-1}
\end{figure}
\begin{proof}
The result is intuitive, but we provide a proof for completeness.

Given $i,j\in[n]$, consider the mapping $\Pi_{i,j}:\R^{n\times n}\rightarrow \R^{n\times n}$ (respectively, $\Pi^{i,j}:\R^{n\times n}\rightarrow\R^{n\times n}$) that transposes the $i$-th and $j$-th rows (respectively, columns) of the matrix.
More precisely, given $V\in\R^{n\times n}$, $V'=\Pi_{i,j}V$ satisfies $V_{ik}'=V_{jk}$, $V_{jk}'=V_{ik}$ and $V_{\ell k}'=V_{\ell k}$ for all $k\in[n]$ and $\ell\in[n]\backslash\{i,j\}$, and $\Pi^{i,j}V=\left(\Pi_{i,j}V^\T\right)^\T$.
Then it is easy to verify that $\|\Pi V\|_{r\rightarrow p}=\|V\|_{r\rightarrow p}$ for $\Pi=\Pi_{i,j}$ or $\Pi=\Pi^{i,j}$ and $M_r(\Pi_{i,j}\Pi^{i,j}V)=M_r(V)$.
It is also immediate that if $V$ satisfies properties (2) and(3), then so does $\Pi_{i,j}V$ and $\Pi^{i,j}V$.

Now, set $V^{(0)}\coloneqq V$, which satisfies properties (1) -- (3), and recursively, for $k\in[m]$, define $i(k),j(k)\in[n]$ to be such that $j(k)>i(k)$ and $\left|V_{i(k),j(k)}^{(2k-2)}\right|=\argmax_{i>2k-2}\left|V_{ij}^{(2k-2)}\right|$, and then
set $V^{(2k-1)}\coloneqq \Pi_{2k-1,i(k)}\Pi^{2k-1,i(k)}V^{(2k-2)}$ and $V^{(2k)}\coloneqq \Pi_{2k,j(k)}\Pi^{2k,j(k)}V^{(2k-1)}$.
Then define $\check V\coloneqq V^{(2m)}$. By the properties of $\Pi_{i,j}$ and $\Pi^{i,j}$ mentioned above, it is immediate that $\check V$ satisfies properties (2), (3) and (6), and by construction, it is clear that $\check V$ is symmetric (since $V$ is) and satisfies (1) and (5). 
Finally, the fact that $\check V$ satisfies (5) and $V$ satisfies (3) together imply that $\check V$ satisfies (4).
\end{proof}

The following result characterizes the fluctuations of $\| A^{\mathsf{la}}\|_{r\rightarrow p}$ and $M_r( A^{\mathsf{la}})$. The latter will be used in \Cref{ss:gro2}.
\begin{proposition}\label{lem:fluct-large}Suppose  
$A_n\in\mathbb R^{n\times n}, n \in \N,$ satisfy \Cref{assume} with index $\alpha$, and let $ A^{\mathsf{la}}_n$ be as in \eqref{a-split} and $b_n$ as in \eqref{d:b}.
Then the following two properties hold: 
\begin{enumerate} 
\item Suppose $1\leq p<r\leq\infty$ and $\alpha\in\left(0,\min\left\{2,\frac{1}{\gamma}\right\}\right)$, where $\gamma\coloneqq \frac{1}{p}-\frac{1}{r}$. Then there exists an $\alpha\gamma$-stable random variable $Y_{\alpha \gamma}$ such that
\begin{align}
\begin{split}
\label{lim:stablepower}
    b_n^{-1}\| A^{\mathsf{la}}_n\|_{r\rightarrow p}\xrightarrow{(d)}Y_{\alpha\gamma}^{\gamma}.
\end{split}
\end{align}
\item Furthermore, suppose that $r>2$ and $\gamma =(r-2)/r$. Then there exists an $\alpha\gamma$-stable random variable such that
\begin{align}
\begin{split}
    b_n^{-1}M_r( A^{\mathsf{la}}_n)\xrightarrow{(d)} Y_{\alpha\gamma}^{\gamma}.
\end{split}
\end{align}
\end{enumerate}
\end{proposition}
\begin{proof}
In this proof, we drop the subscript $n$ of $A_n$ and $ A^{\mathsf{la}}_n$ for brevity. Since $A$ is symmetric, \eqref{a-split} ensures $ A^{\mathsf{la}}$ is also symmetric and by \Cref{lemma:typicalBehaviour}, $ A^{\mathsf{la}}$ satisfies properties (1) and (3) of \Cref{modification}. Moreover, if $2m(n)$ denotes the number of nonzero entries of $ A^{\mathsf{la}}= A^{\mathsf{la}}_n$, then an application of \Cref{lemma:num-large-entry} with $\theta=1-\frac{\zeta}{2}$, together with the choice of $\zeta$ in \eqref{def:eta-para}, shows that $m(n)=o(n)$ \whp. Thus, \Cref{modification} shows that \whp, there exists another symmetric matrix $\check A$ that satisfies properties (1) -- (6) of \Cref{modification}.

\begin{enumerate}
\item Fix $1\leq p<r\leq \infty$. By the structure of $\check A$ and the corresponding dual formulation \eqref{eq:rtop-dual} of $\|\check A\|_{r\rightarrow p}$, we see that,
\begin{align}
\begin{split}
    \|\check A\|_{r\rightarrow p}=\max_{\|\bm x\|_r =\|\bm y\|_{p^*} =1}\sum_{k\in[m]}\left(\check A_{2k-1,2k}x_{2k}y_{2k-1}+\check A_{2k,2k-1}x_{2k-1}y_{2k}\right).
    \label{rewrite2}
\end{split}
\end{align}
For all $\bm x,\bm y\in\mathbb R^n$ with $\|\bm x\|_r=\|\bm y\|_{p^*}=1$, using H\"older's inequality twice, we have
\begin{align}\label{seq2'}
\begin{split}
    &\sum_{k\in[m(n)]}\left(\check A_{2k-1,2k}x_{2k}y_{2k-1}+\check A_{2k,2k-1}x_{2k-1}y_{2k}\right)\\
    &\leq \left(2\sum_{k\in[m(n)]}\left|\check A_{2k-1,2k} \right|^{\frac{1}{\gamma}}\right)^{\gamma}\left[\sum_{k\in[m(n)]}\left(|x_{2k}y_{2k-1}|^{\frac{rp}{rp-r+p}}+|x_{2k-1}y_{2k}|^{\frac{rp}{rp-r+p}}\right)\right]^{\frac{rp-r+p}{rp}}\\
    &\leq \left(2\sum_{k\in[m(n)]}\left|\check A_{2k-1,2k} \right|^{\frac{1}{\gamma}}\right)^{\gamma}\left(\sum_{k\in[2m(n)]}|x_k|^r\right)^{\frac{1}{r}}\left(\sum_{k\in[2m(n)]}|y_k|^{p^*}\right)^{\frac{1}{p^*}}\\
    &\leq \left(2\sum_{k\in[m(n)]}\left|\check A_{2k-1,2k} \right|^{\frac{1}{\gamma}}\right)^{\gamma},
\end{split}
\end{align}
and the equalities can be achieved simultaneously by choosing vectors $\bm x=(x_1,\ldots,x_n),\bm y=(y_1,\ldots,y_n)$ that satisfy $\|\bm x\|_r=\|\bm  y\|_{p^*}=1$, and with $\psi_r$ as defined in \eqref{def:psi}, 
\begin{align}\label{opt-x}
\begin{split}
    x_i \propto \begin{cases}\psi_{\frac{r}{r-p}}\left(\check A_{2 k-1,2 k}\right), & \mbox{ if } i=2 k-1,1 \leq k \leq m (n)\\ \left|\check A_{2 k-1,2 k}\right|^{\frac{p}{r-p}}, & \mbox{ if } i=2 k, 1 \leq k \leq m(n) \\ 0, & \mbox{ if } 2 m(n)+1 \leq i \leq n,\end{cases}
\end{split}
\end{align}
and 
\begin{align}\label{opt-y}
\begin{split}
    y_i \propto \begin{cases}\psi_{\frac{r p-p}{r-p}}\left(\check A_{2 k-1,2 k}\right), & i=2 k-1,1 \leq k \leq m(n) \\ \left|\check A_{2 k-1,2 k}\right|^{\frac{r p-r}{r-p}}, & i=2 k, 1 \leq k \leq m(n) \\ 0, & 2 m(n)+1 \leq i \leq N, \end{cases}
\end{split}
\end{align}

Due to the equality $\left\{\check A_{ij}:\,i,j\in[n]\right\}=\left\{ A^{\mathsf{la}}_{ij}:\,i,j\in[n]\right\}$, properties (1)-(3) of \Cref{modification}, and the definition of $ A^{\mathsf{la}}$ in \eqref{a-split}, we have \whp,
\begin{align}\label{seq4}
\begin{split}
    2\sum_{k\in[m(n)]}\left|\check A_{2k-1,2k} \right|^{\frac{1}{\gamma}}
    =\sum_{i,j\in[n]}\left|\check A_{ij}\right|^{\frac{1}{\gamma}}
    =\sum_{i,j\in[n]}\left| A^{\mathsf{la}}_{ij} \right|^{\frac{1}{\gamma}}
    =2\sum_{1\leq i<j\leq n}|a_{ij}|^{\frac{1}{\gamma}}\bbm 1\left(|a_{ij}|\geq n^{\frac{2-\zeta}{\alpha}}\right).
\end{split}
\end{align}

Since $\alpha \in (0,\frac{1}{\gamma})$,
applying \Cref{l:sum-small} with $\left\{X_\ell\right\}_{\ell\in[n(n+1)/2]}=\left\{|a_{ij}|\right\}_{1\leq i\leq j\leq n}$, $q=\frac{1}{\gamma}$ and $\beta=\frac{2-\zeta}{\alpha}$,  we have \whp,
\begin{align}\label{seq3'}
\begin{split}
    \sum_{1\leq i\leq j\leq n}|a_{ij}|^{\frac{1}{\gamma}}\bbm 1\left\{|a_{ij}|\leq n^{\frac{2-\zeta}{\alpha}}\right\}=O_\delta\left(n^{\frac{2}{\alpha\gamma}-\zeta\left(\frac{1}{\alpha\gamma}-1\right)}\right)=o\left(b_n^{\frac{1}{\gamma}}\right), 
\end{split}
\end{align}
where $b_n$ is as in \eqref{d:b}. 
The generalized central limit theorem \cite[Theorem~3.8.2]{Durrett19} then implies
\begin{align}\label{seq5}
\begin{split}
    b_n^{-\frac{1}{\gamma}}\cdot2\sum_{1\leq i<j\leq n}|a_{ij}|^{\frac{1}{\gamma}}\xrightarrow{(d)} Y_{\alpha\gamma},
\end{split}
\end{align}
where $Y_{\alpha\gamma}$ is an $\alpha\gamma$-stable random variable (see \Cref{def:stable}). 
The limit \eqref{lim:stablepower} stated in property 1 then follows on 
combining \eqref{rewrite2}, \eqref{seq2'}, \eqref{seq4}, \eqref{seq3'} and \eqref{seq5} with the continuous mapping theorem. 

\item Fix $2<r\leq\infty$. 
Observe that when $p=r^*$, the optimizers $\bm x,\bm y$ for $\|\check A\|_{r\rightarrow r^*}$, satisfying $\|\bm x\|_r=\|\bm y\|_{p^*}=1$, \eqref{opt-x} and \eqref{opt-y}, coincide. This implies that 
\begin{align}\label{comp2}
\begin{split}
    \|\check A\|_{r\rightarrow r^*}=\bm x^\T \check A\bm x\leq M_r(\check A).
\end{split}
\end{align}
When combined with the converse inequality that always holds by \eqref{ineq:dual}, it follows that 
$M_r(\check A)=\|\check A\|_{r\rightarrow r^*}$. Thus, the conclusion follows from the first part of the proposition and the identity 
$M_r(\check A) = M_r( A^{\mathsf{la}})$ 
stated as property (6) of \Cref{modification}. \qedhere
\end{enumerate}
\end{proof}
We are now ready to present the proof of \Cref{thm:rtop2}.
\begin{proof}[Proof of \Cref{thm:rtop2} and \Cref{thm:rtop2'}]
Recall the decomposition of $A= A^{\mathsf{sm}}+ A^{\mathsf{int}}+ A^{\mathsf{la}}$ in \eqref{a-split}. By the triangle inequality, \Cref{lem:omit-small} and \Cref{lem:omit-middle}, under assumptions of \Cref{thm:rtop2} or \Cref{thm:rtop2'}, we have \whp,
\begin{align}\label{approx1}
\begin{split}
    \left|\|A\|_{r\rightarrow p}-\| A^{\mathsf{la}}\|_{r\rightarrow p}\right|\leq \| A^{\mathsf{sm}}\|_{r\rightarrow p}+\| A^{\mathsf{int}}\|_{r\rightarrow p}=o(b_n).
\end{split}
\end{align}
Together with \Cref{lem:fluct-large}, this completes the proof.
\end{proof}

\section{Proof of \texorpdfstring{\Cref{thm:rtop3}}{non-centered}}\label{s:rtop3}
Throughout this section, denote by $\gamma=\frac{r-p}{rp}$ as in the statement of \Cref{thm:rtop3}. Fix $\alpha\in\left(0,\min\left\{2,\frac{1}{\gamma}\right\}\right)$ and $1<p<r<\infty$, define 
\begin{equation}\label{eta-para2}
\begin{aligned}
\eta & \coloneqq \frac{1}{2}\left(1-\alpha\gamma\right), \\
\zeta & \coloneqq \min \left\{\frac{1}{4}\left(1-\alpha\gamma\right), \left(1-\frac{\alpha}{2}\right)\frac{p-1}{r}\right\}.
\end{aligned}
\end{equation}
With these $\eta$ and $\zeta$, we adopt the same decomposition scheme of $A$ as in \eqref{a-split}. Also note that under the assumptions of \Cref{thm:rtop3}, $0<\zeta<\eta$.

The section is organized as follows. \Cref{ss:aux} contains some auxiliary results used in the derivation of upper and lower bounds on fluctuation of the $r\rightarrow p$ norm of $\mu\bm 1\bm 1^\T + A^{\mathsf{la}}+ A^{\mathsf{int}}$, which are presented in \Cref{ss:main-upper} and \Cref{ss:main-lower} respectively.  The proof of \Cref{thm:rtop3} is completed in \Cref{ss:matching}  by showing that these two bounds match asymptotically and that adding back $ A^{\mathsf{sm}}$ leads to a negligible change in the asymptotic fluctuations.

\subsection{Auxiliary results}\label{ss:aux}
We start with four auxiliary lemmas. The first two concern the structures of matrix $ A^{\mathsf{int}}+ A^{\mathsf{la}}$ and $ A^{\mathsf{la}}$.
\begin{lemma}\label{l:product0} Suppose the assumptions of \Cref{thm:rtop3} hold and
$ A^{\mathsf{la}}=\left( a^{\mathsf{la}}_{ij}\right)_{i,j\in[n]} , A^{\mathsf{int}}=\left( a^{\mathsf{int}}_{ij}\right)_{i,j\in[n]}$ are as in \eqref{a-split}. Then \whp, for all $i,j,k\in[n]$,
\begin{align}
\begin{split}
     a^{\mathsf{la}}_{ij} a^{\mathsf{int}}_{ik}=0.
\end{split}
\end{align}
\end{lemma}
\begin{proof}
By \eqref{eta-para2}, \eqref{a-split}, the assumption that $(a_{ij})$ are heavy-tailed with index $\alpha$ and \Cref{d:heavy}, the probability that there exists $i,j,k\in[n]$ such that $ a^{\mathsf{la}}_{ij} a^{\mathsf{int}}_{ik}\neq 0$ is upper bounded by
\begin{align}
\begin{split}
    n\binom{n}{2}L\left(n^{\frac{2-\zeta}{\alpha}}\right)L\left(n^{\frac{1+\eta}{\alpha}}\right)\left(n^{\frac{2-\zeta}{\alpha}}\right)^{-\alpha}\left(n^{\frac{1+\eta}{\alpha}}\right)^{-\alpha}\leq n^{(-\eta+\zeta)/2}\rightarrow 0.
\end{split}
\end{align}
The proof is complete.
\end{proof}
\begin{lemma} Suppose the assumptions of \Cref{thm:rtop3} hold and
$ A^{\mathsf{la}}=\left( a^{\mathsf{la}}_{ij}\right)_{i,j\in[n]} $ is as in \eqref{a-split}. Fix $q\geq 0$.
Then \whp, for all $i\in[n]$,
\begin{align}\label{eq-01}
\begin{split}
    \left|( A^{\mathsf{la}}\bm 1)_i\right|^q=\sum_{j\in[n]}| a^{\mathsf{la}}_{ij}|^q.
\end{split}
\end{align}
\end{lemma}
\begin{proof}
Recalling the choice of $\zeta$ in \eqref{eta-para2}, definition of $ A^{\mathsf{la}}$ in \eqref{a-split} and \eqref{eqn:order-bn}, \Cref{lemma:typicalBehaviour} (b) implies that \whp, there is at most one nonzero entry on each row or column of $ A^{\mathsf{la}}$. Say the nonzero entry on the $i$-th row of $ A^{\mathsf{la}}$ is on the $i^*$-column. 
We have \whp, for all $i\in[n]$,
\begin{align}
\begin{split}
    \left|\left( A^{\mathsf{la}}\bm 1\right)_i\right|^q=|a_{ii^*}|^q=\sum_{j\in[n]}\left| a^{\mathsf{la}}_{ij}\right|^q.
\end{split}
\end{align}
This completes the proof.
\end{proof}
The third result identifies a typical property of heavy-tailed entries.
\begin{lemma}\label{l:sum-big}
Let $\left(X_i\right)_{i\in[n(n+1)/2]}$ be i.i.d. copies of heavy-tailed random variable with index $\alpha\in(0,2)$. Fix $q>0$ and $\beta\in(0,\frac{2}{\alpha})$. Denote by $\tilde X_i=X_i\bbm 1\left\{|X_i|> n^\beta\right\}$. Then \whp,
\begin{align}\label{goalgoal2}
\begin{split}
    \sum_{i\in[n(n+1)/2]}\left|\tilde X_i\right|^q =O_\delta\left(n^{\max\left\{q\beta-\alpha\beta+2,\frac{2q}{\alpha}\right\}}\right).
\end{split}
\end{align}
\end{lemma}
\begin{proof}
By considering $\left|X_i\right|$, we may, without loss of generality, assume that $X_i$'s are nonnegative. Fix any $\epsilon>0$ and $\theta$ such that 
\begin{align}\label{theta-para}
\begin{split}
    \max\left\{\frac{\alpha\beta}{2},1-\epsilon\right\}<\theta<1.
\end{split}
\end{align}
Consider the following decomposition of $\tilde X_i$:
\begin{align}
\begin{split}
    \tilde X_i = \bar X_i+\hat X_i,
\end{split}
\end{align}
where 
\begin{align}\label{x-split}
\begin{split}
    \bar X_i\coloneqq \tilde X_i\bbm 1\left\{\tilde X_i\leq n^{\frac{2\theta}{\alpha}}\right\},\quad \hat X_i\coloneqq \tilde X_i\bbm 1\left\{\tilde X_i>n^{\frac{2\theta}{\alpha}}\right\}.
\end{split}
\end{align}
By \Cref{lemma:num-large-entry}, \whp, there are $\Theta_\delta\left({n^{2-2\theta}}\right)$ nonzero $\hat X_i$. By \eqref{ccdf} and \Cref{l:slowly-varying-property}, for any $\epsilon>0$,
\begin{align}
\begin{split}
    \P\left(\max_{i\in[n(n+1)/2]}\hat X_i\geq n^{\frac{2}{\alpha}+\epsilon}\right)\leq \P\left(\max_{i\in[n(n+1)/2]}X_i\geq n^{\frac{2}{\alpha}+\epsilon}\right)=O_\delta\left(n^2\cdot n^{-2-\alpha\epsilon}\right)=o(1),
\end{split}
\end{align}
and hence $\max_{i\in[n(n+1)/2]}\hat X_i=O_\delta\left(n^{\frac{2}{\alpha}}\right)$. 
Therefore \whp, for any $q>0$,
\begin{align}\label{hat-x}
\begin{split}
    \sum_{i\in[n(n+1)/2]}\hat X_i^{q}=O_\delta\left(n^{2-2\theta+\frac{2q}{\alpha}}\right)=O_\delta\left(n^{2\epsilon+\frac{2q}{\alpha}}\right),
\end{split}
\end{align}
where we used \eqref{theta-para} in the last step.

On the other hand, we have by \eqref{theta-para},
\begin{equation}\label{mean-var-bar-x}
\begin{aligned}
\E\left[\sum_{i\in[n(n+1)/2]}\bar X_i^q\right]
&=\frac{n(n+1)}{2}\E\left[\bar X_1^q\right]
&\leq n^{\epsilon/2+2} \int_{n^{\beta q}}^{n^{2\theta q/\alpha}}x^{-\frac{\alpha}{q}}dx
&=n^{2+\epsilon+\max\left\{\beta (q-\alpha),\frac{2\theta}{\alpha}\left(q-\alpha\right)\right\}},\\
\operatorname{Var}\left(\sum_{i\in[n(n+1)/2]}\bar X_i^q\right)
&\leq \frac{n(n+1)}{2}\E\left[\bar X_1^{2q}\right]
&=n^{\epsilon/2+2}\int_{n^{2q\beta}}^{n^{4\theta q/\alpha}}x^{-\frac{\alpha}{2q}}dx
&=n^{2+\epsilon+\max\left\{\beta(2q-\alpha),\frac{2\theta}{\alpha}\left(2q-\alpha\right)\right\}}.
\end{aligned}
\end{equation}
Combining \eqref{mean-var-bar-x} with Markov's inequality and \eqref{theta-para}, we have \whp, for all $q\geq0$,
\begin{align}\label{bar-x}
\begin{split}
    \sum_{i\in[n(n+1)/2]}\bar X_i^q=O_\delta\left(n^{2+\epsilon +\max\left\{\beta(q-\alpha),\frac{2\theta}{\alpha}(q-\alpha)\right\}}\right).
\end{split}
\end{align}
Combining $\theta<1$ from \eqref{theta-para}, \eqref{hat-x} and \eqref{bar-x}, we have \whp, for any $q>0$,
\begin{align}
\begin{split}
    \sum_{i\in[n(n+1)/2]}\tilde X_i^q=O_\delta\left(n^{\max\left\{2\epsilon+\frac{2q}{\alpha},2+\epsilon+\beta (q-\alpha)\right\}}\right).
\end{split}
\end{align}
Recalling that $\epsilon>0$ is arbitrary, the proof is complete.
\end{proof}
The fourth lemma, which we will use multiple times in the rest of the section, is an immediate consequence of \Cref{l:sum-small} and \Cref{l:sum-big},
\begin{lemma}
Suppose the assumptions of \Cref{thm:rtop3} hold and recall that $\gamma=\frac{r-p}{rp}$. With $b_n$ as defined in \eqref{d:b} and $\zeta$ in \eqref{eta-para2}, we have \whp,
\begin{align}
    \sum_{i,j\in[n]}|a_{ij}|^{\frac1\gamma}\bbm 1\left(|a_{ij}|\leq n^{\frac{2-\zeta}{\alpha}}\right)&=o\left(b_n^{\frac1\gamma}\right),\label{rp-sum-small}\\
    \sum_{i,j\in[n]}|a_{ij}|^{\frac1\gamma}&=o\left(n^{1+\frac1\gamma}\right).\label{rp-sum}
\end{align}
\end{lemma}
\begin{proof}
Note that \eqref{rtop3-cond} implies $\alpha<\frac{1}{\gamma}$. Hence, \Cref{l:sum-small} with $\{X_i:\,i\in[n(n+1)/2]\}=\{a_{ij}:\,1\leq i\leq j\leq n\},\, q=\frac{1}{\gamma}$ and $\beta=\frac{2-\zeta}{\alpha}$ implies that \whp, 
\begin{align}
\begin{split}
    \sum_{i,j\in[n]}|a_{ij}|^{\frac1\gamma}\bbm 1\left(|a_{ij}|\leq n^{\frac{2-\zeta}{\alpha}}\right)=O_\delta\left(n^{\frac{2}{\alpha\gamma}-\left(\frac{1}{\alpha\gamma}-1\right)\zeta}\right)=o\left(b_n^{\frac1\gamma}\right),
\end{split}
\end{align}
where we use \eqref{eqn:order-bn} in the last equality. This proves \eqref{rp-sum-small}. 
Applying \Cref{l:sum-big} with $\{X_i:\,i\in[n(n+1)/2]\}=\{a_{ij}:\,1\leq i\leq j\leq n\},\, q=\frac{1}{\gamma}>\alpha$ and $\beta=\frac{2-\zeta}{\alpha}$, we have \whp,
\begin{align}
\begin{split}
    \sum_{i,j\in[n]}|a_{ij}|^{\frac1\gamma}\bbm 1\left(|a_{ij}|>n^{\frac{2-\zeta}{\alpha}}\right)=O_\delta\left(n^{\max\left\{\frac{2}{\alpha\gamma}-\left(\frac{1}{\alpha\gamma}-1\right)\zeta,\,\frac{2}{\alpha\gamma}\right\}}\right)=O_\delta\left(n^{\frac{2}{\alpha\gamma}}\right)=o\left(n^{1+\frac1\gamma}\right),
\end{split}
\end{align}
where we use \eqref{rtop3-cond} in the second and third equalities. Together with \eqref{eqn:order-bn} and \eqref{rp-sum-small}, this  proves \eqref{rp-sum}. 
\end{proof}

\subsection{Upper bound}\label{ss:main-upper}
The goal of this section is the following lemma.
\begin{proposition}\label{l:upper}
Under the assumptions of \Cref{thm:rtop3}, recall that $\gamma=\frac{r-p}{rp}$, we have \whp
\begin{align}\label{eqn:upper}
\begin{split}
    \|\mu \mathbf{1} \mathbf{1}^{\top}+ A^{\mathsf{la}}+ A^{\mathsf{int}}\|_{r \rightarrow p} \leq n^{1+\gamma} |\mu|+\gamma |\mu|^{1-\frac{1}{\gamma}} n^{\gamma-\frac{1}{\gamma}} \sum_{i, j=1}^n\left|a_{i j}\right|^{\frac{1}{\gamma}}+o\left(n^{\gamma-\frac{1}{\gamma}} b_n^{\frac{1}{\gamma}}\right).
\end{split}
\end{align}
\end{proposition}

\begin{proof}In this proof, we will use $C$ to denote a positive constant that may depend on $\alpha, r, p$, and $\mu$ and whose value may vary from one occurrence to another.  Since $\|\mu\bm 1\bm 1^\T+ A^{\mathsf{int}}+ A^{\mathsf{la}}\|_{r\rightarrow p}\leq \| |\mu|\bm 1\bm 1^\T +| A^{\mathsf{la}}|+| A^{\mathsf{int}}|\|_{r\rightarrow p}$, it suffices to prove the lemma in the case that $\mu>0$ and that $ A^{\mathsf{la}}, A^{\mathsf{int}}$ contain only nonnegative entries. 
By \Cref{lem:new-upper-bound}, we have
\begin{align}\label{d:upper}
\begin{split}
    \|\mu\bm 1\bm 1^\T+ A^{\mathsf{int}}+ A^{\mathsf{la}}\|_{r\rightarrow p}\leq \mathcal U(A)\coloneqq  \left(\bm 1^\T \left((\mu\bm 1\bm 1^\T+ A^{\mathsf{int}}+ A^{\mathsf{la}})\bm 1\right)^{\frac{1}{\gamma}}\right)^{\gamma}.
\end{split}
\end{align}
Hence it suffices to show that \whp,
\begin{align}\label{upper-goal}
\begin{split}
    \mathcal U(A)\leq n^{1+\gamma} |\mu|+\gamma |\mu|^{1-\frac{1}{\gamma}} n^{\gamma-\frac{1}{\gamma}} \sum_{i, j=1}^n\left|a_{i j}\right|^{\frac{1}{\gamma}}+o\left(n^{\gamma-\frac{1}{\gamma}} b_n^{\frac{1}{\gamma}}\right).
\end{split}
\end{align}
To show this, we first claim that, we have \whp,
\begin{align}\label{upper-claim}
\begin{split}
    \left(\left(\mu\bm 1\bm 1^\T + A^{\mathsf{int}}+ A^{\mathsf{la}}\right)\bm 1\right)^{\frac{rp}{r-p}}&\leq (n\mu)^{\frac{1}{\gamma}}\bm 1+\left( A^{\mathsf{la}}\bm 1\right)^{\frac1\gamma}\\
    &\quad +Cn\left( A^{\mathsf{la}}\bm 1\right)^{\frac1\gamma-1}+Cn^{\frac1\gamma-1} A^{\mathsf{int}}\bm 1+C\left( A^{\mathsf{int}}\bm 1\right)^{\frac1\gamma},
\end{split}
\end{align}
where the inequality stands for entrywise comparison. Let $i\in[n]$ be fixed. We prove \eqref{upper-claim} by considering two cases.
\begin{enumerate}[(1)]
\item Suppose $( A^{\mathsf{la}}\bm 1)_i\neq 0$. Then by \Cref{l:product0}, we have \whp, $( A^{\mathsf{int}}\bm 1)_i=0$. By the definition of $ A^{\mathsf{la}}$ in \eqref{a-split} and the choice of $\zeta$ in \eqref{eta-para2}, it follows that
\begin{align}
\begin{split}
    ( A^{\mathsf{la}}\bm 1)_i\geq n^{\frac{2-\zeta}{\alpha}}\gg n.
\end{split}
\end{align}
Using the elementary inequality $(x+y)^q\leq x^q+Cx^{q-1}y,\,\forall x,q>0,\,0\leq y\leq \frac{x}{2}$, we have \whp,
\begin{align}\label{upper-1}
\begin{split}
    \left(\left(\mu\bm 1\bm 1^\T+ A^{\mathsf{int}}+ A^{\mathsf{la}}\right)\bm 1\right)_i^{\frac1\gamma}=\left(n\mu+\left( A^{\mathsf{la}}\bm 1\right)_i\right)^{\frac1\gamma}\leq \left( A^{\mathsf{la}}\bm 1\right)_{i}^{\frac1\gamma}+Cn\left( A^{\mathsf{la}}\bm 1\right)^{\frac1\gamma-1}.
\end{split}
\end{align}
\item Suppose $( A^{\mathsf{la}}\bm 1)_i=0$. Then by the elementary inequality $(x+y)^q\leq x^q+Cx^{q-1}y+Cy^q,\ \forall\,x,q>0,y\geq 0$, we have
\begin{align}\label{upper-2}
\begin{split}
    \left(\left(\mu\bm 1\bm 1^\T + A^{\mathsf{int}}+ A^{\mathsf{la}}\right)\bm 1\right)_i^{\frac1\gamma}&=\left(n\mu+\left( A^{\mathsf{int}}\bm 1\right)_i\right)^{\frac1\gamma}\\
    &\leq (n\mu)^{\frac1\gamma}+Cn^{\frac1\gamma-1}\left( A^{\mathsf{int}}\bm 1\right)_i+C\left( A^{\mathsf{int}}\bm 1\right)_i^{\frac1\gamma}.
\end{split}
\end{align}
\end{enumerate}
The claim \eqref{upper-claim} is complete upon combining \eqref{upper-1} and \eqref{upper-2}.

By \eqref{sparse-consequence}, \eqref{eq-01} and \eqref{upper-claim}, we have \whp,
\begin{align}\label{upper-inner}
\begin{split}
    \bm 1^\T \left(\left(\mu\bm 1\bm 1^\T+ A^{\mathsf{int}}+ A^{\mathsf{la}}\right)\bm 1\right)^{\frac1\gamma}\leq n^{\frac1\gamma+1}\mu^{\frac1\gamma}+\sum_{i,j\in[n]} (a^{\mathsf{la}}_{ij})^{\frac1\gamma}+\sum_{i=1}^3 R_i,
\end{split}
\end{align}
where 
\begin{align}
\begin{split}
    R_1\coloneqq Cn\sum_{i,j\in[n]} (a^{\mathsf{la}}_{ij})^{\frac1\gamma-1},
    \quad R_2\coloneqq Cn^{\frac1\gamma-1}\sum_{i,j\in[n]} a^{\mathsf{int}}_{ij}, 
    \quad R_3\coloneqq C\sum_{i,j\in[n]}( a^{\mathsf{int}}_{ij})^{\frac1\gamma}.
\end{split}
\end{align}
We next claim that \whp,
\begin{align}\label{upper-claim-2}
\begin{split}
    R_i=o\left(b_n^{\frac1\gamma}\right) ,\quad i=1,2,3.
\end{split}
\end{align}
The verification is as follows.
\begin{enumerate}[(1)]
\item By \Cref{l:sum-big} with $\{X_i:\,i\in[n(n+1)/2]\}=\{a_{ij}:\, 1\leq i\leq j\leq n\},\ q=\frac1\gamma-1$ and $\beta=\frac{2-\zeta}{\alpha}$, we have \whp,
\begin{align}
\begin{split}
    \sum_{i,j\in[n]} (a^{\mathsf{la}}_{ij})^{\frac1\gamma-1}&=\sum_{i,j\in[n]}a_{ij}^{\frac1\gamma-1}\bbm 1\left(a_{ij}\geq n^{\frac{2-\zeta}{\alpha}}\right)\\
    &=O_\delta\left(n^{\max\left\{\left(\frac{1}{\gamma}-1\right) \frac{2-\zeta}{\alpha}+\zeta,\,\left(\frac{1}{\gamma}-1\right) \frac{2}{\alpha}\right\}}\right)\\
    &=O_\delta\left(n^{\left(\frac1\gamma-1\right)\frac{2}{\alpha}+\zeta}\right).
\end{split}
\end{align}
Together with the choice of $\zeta$ in \eqref{eta-para2}, this implies that \whp,
\begin{align}
\begin{split}
    R_1=O_\delta\left(n^{\frac{2}{\alpha\gamma}+1-\frac{2}{\alpha}+\zeta}\right)=o\left(b_n^{\frac{1}{\gamma}}\right).
\end{split}
\end{align}
\item By \Cref{l:sum-big} with $\{X_i:\,i\in[n(n+1)/2]\}=\{a_{ij}:\,1\leq i\leq j\leq n\},\,q=1$ and $\beta=\frac{1+\eta}{\alpha}$, we have \whp,
\begin{align}
\begin{split}
    \sum_{i,j\in[n]} a^{\mathsf{int}}_{ij}&\leq\sum_{i,j\in[n]}a_{ij}\bbm 1\left(a_{ij}\geq n^{\frac{1+\eta}{\alpha}}\right)\\
    &=O_\delta\left(n^{\max\left\{\frac{1}{\alpha}+1+\eta\left(\frac{1}{\alpha}-1\right),\,\frac{2}{\alpha}\right\}}\right)\\
    &=O_\delta\left(n^{\frac{1}{\alpha}+1}\right),
\end{split}
\end{align}
where we use \eqref{rtop3-cond} in the last equality. Using \eqref{rtop3-cond} again, this implies that \whp,
\begin{align}
\begin{split}
    R_2=O_\delta\left(n^{\frac1\gamma+\frac{1}{\alpha}}\right)=o\left(b_n^{\frac1\gamma}\right).
\end{split}
\end{align}
\item By \eqref{rp-sum-small}, we have \whp, 
\begin{align}
\begin{split}
    R_3=C\sum_{i,j\in[n]}a_{ij}\bbm 1\left(a_{ij}\leq n^{\frac{2-\zeta}{\alpha}}\right) =o\left(b_n^{\frac1\gamma}\right).
\end{split}
\end{align}
\end{enumerate}
Combining \eqref{upper-inner}, \eqref{upper-claim-2}, \eqref{a-split} and \eqref{rp-sum-small}, we have \whp,
\begin{align}
\begin{split}
    \mathbf{1}^{\top}\left(\left(\mu \mathbf{1} \mathbf{1}^{\top}+ A^{\mathsf{int}}+ A^{\mathsf{la}}\right) \mathbf{1}\right)^{\frac{1}{\gamma}}&\leq n^{\frac1\gamma+1}\mu^{\frac1\gamma}+\sum_{i,j\in[n]} (a^{\mathsf{la}}_{ij})^{\frac1\gamma}+o\left(b_n^{\frac1\gamma}\right)\\
    &=n^{\frac1\gamma+1}\mu^{\frac1\gamma}+\sum_{i,j\in[n]}a_{ij}^{\frac1\gamma}+o\left(b_n^{\frac1\gamma}\right).
\end{split}
\end{align}
Together with \eqref{rp-sum}, the definition of $\mathcal U(A)$ in \eqref{d:upper}, the fact that $\gamma=\frac{r-p}{rp}\in(0,1]$, and the elementary inequality $(x+y)^{q}\leq x^q+q x^{q-1}y$ for all $x>0,\,q\in(0,1]$ and $y\geq 0$, this implies that \whp, 
\begin{align}
\begin{split}
    \mathcal U(A)&\leq \left(n^{\frac1\gamma+1}\mu^{\frac1\gamma}+\sum_{i,j\in[n]}a_{ij}^{\frac1\gamma}+o\left(b_n^{\frac1\gamma}\right)\right)^{\gamma}\\
    &\leq n^{1+\gamma}\mu+\gamma n^{\gamma-\frac{1}{\gamma}}\mu^{1-\frac{1}{\gamma}}\sum_{i,j\in[n]}a_{ij}^{\frac{1}{\gamma}}+o\left(n^{\gamma-\frac{1}{\gamma}}b_n^{\frac{1}{\gamma}}\right).
\end{split}
\end{align}
This establishes \eqref{upper-goal}, and hence completes the proof of the lemma.
\end{proof}

\subsection{Lower bound}\label{ss:main-lower}
This section is devoted to the proof of the following proposition.
\begin{proposition}\label{l:lower}
Under the assumptions of \Cref{thm:rtop3}, recall that $\gamma=\frac{r-p}{rp}$, we have \whp
\begin{align}
\begin{split}
    \|\mu\bm 1\bm 1^\T+ A^{\mathsf{int}}+ A^{\mathsf{la}}\|_{r\rightarrow p}\geq n^{1+\gamma} |\mu|+\gamma |\mu|^{1-\frac{1}{\gamma}} n^{\gamma-\frac{1}{\gamma}} \sum_{i, j=1}^n\left|a_{i j}\right|^{\frac{1}{\gamma}}+o\left(n^{\gamma-\frac{1}{\gamma}} b_n^{\frac{1}{\gamma}}\right).
\end{split}
\end{align}
\end{proposition}
 
 The  ansatz for this lower bound was arrived at by taking inspiration from the nonlinear power method; see \Cref{rem:lbound} for further details.  First, in  
 \Cref{s:proof-lbound}, we present its proof.

\subsubsection{Proof of the lower bound}
\label{s:proof-lbound}

Define $\hat{\bm x}=(\hat x_1,\ldots,\hat x_n),\hat{\bm y}=(\hat y_1,\ldots,\hat y_n)$ such that, for $i\in[n]$,
\begin{align}\label{hatxy}
\begin{split}
    &\hat x_i\coloneqq(n |\mu|)^{\frac{p}{r-p}}\sgn(\mu)+\sum_{j:\,j<i}|  a^{\mathsf{la}}_{ij}|^{\frac{p}{r-p}}+\sum_{j:\,j>i}\psi_{\frac{r}{r-p}}( a^{\mathsf{la}}_{ij}), \\
    &\hat y_i\coloneqq(n |\mu|)^{\frac{r(p-1)}{r-p}}+\sum_{j:\,j<i}|  a^{\mathsf{la}}_{ij}|^{\frac{r(p-1)}{r-p}}+\sum_{j:\,j>i}\psi_{\frac{p(r-1)}{r-p}}( a^{\mathsf{la}}_{ij}),
\end{split}
\end{align}
where $\psi_{\frac{r}{r-p}},\psi_{\frac{p(r-1)}{r-p}}$ are the functionals defined in \eqref{def:psi}. 
Consider
\begin{align}\label{d:lower}
\begin{split}
    \mathcal L(A)\coloneqq \frac{\mathcal L_1(A)}{\mathcal L_2(A)},
\end{split}
\end{align}
where
\begin{align}\label{l1-l2}
\begin{split}
    &\mathcal L_1(A)\coloneqq \hat{\bm y}^\T\left(\mu\bm 1\bm 1^\T+ A^{\mathsf{int}}+ A^{\mathsf{la}}\right)\hat{\bm x},\\
    &\mathcal L_2(A)\coloneqq \|\hat{\bm y}\|_{p^*}\|\hat{\bm x}\|_r.
\end{split}
\end{align}
By the dual formulation \eqref{eq:rtop-dual} of the $r\rightarrow p$ norm,
we have
\begin{align}
\begin{split}
    \mathcal L(A)\leq \|\mu\bm 1\bm 1^\T+ A^{\mathsf{int}}+ A^{\mathsf{la}}\|_{r\rightarrow p}.
\end{split}
\end{align}
To prove \Cref{l:lower}, it suffices to establish the following lemma.
\begin{lemma}\label{l:lower1}
Under the assumptions of \Cref{thm:rtop3}, recall that $\gamma=\frac{r-p}{rp}$, we have \whp,
\begin{align}
\begin{split}
    \mathcal L(A)\geq n^{1+\gamma} |\mu|+\gamma |\mu|^{1-\frac{1}{\gamma}} n^{\gamma-\frac{1}{\gamma}} \sum_{i, j=1}^n\left|a_{i j}\right|^{\frac{1}{\gamma}}+o\left(n^{\gamma-\frac{1}{\gamma}} b_n^{\frac{1}{\gamma}}\right).
\end{split}
\end{align}
\end{lemma}
\begin{proof}In this proof, we will use $C$ to denote a positive constant that may depend on $\alpha, r, p$, and $\mu$ and whose value may vary from one occurrence to another. Observe that since $\|\mu\bm 1\bm 1^\T+ A^{\mathsf{int}}+ A^{\mathsf{la}}\|_{r\rightarrow p}=\|-\mu\bm 1\bm 1^\T- A^{\mathsf{int}}- A^{\mathsf{la}}\|_{r\rightarrow p}$, we may assume $\mu>0$ for the rest of the proof. We proceed in three steps. 

\noindent\emph{Step 1: Estimation of $\mathcal L_1(A)$.}
Expanding the definition of $\mathcal L_1$ in \eqref{l1-l2}, we see that
\begin{align}\label{eq501}
\begin{split}
    \mathcal L_1(A)=\mu\bm 1^\T\hat{\bm x}\hat{\bm y}^\T \bm 1 + \hat{\bm y}^\T A^{\mathsf{int}}\hat{\bm x}+\hat{\bm y}^\T A^{\mathsf{la}}\hat{\bm x}.
\end{split}
\end{align}
By the definitions of $\hat{\bm x},\hat{\bm y}$ in \eqref{hatxy}, we have \whp,
\begin{align}\label{eq502}
\begin{split}
    \mu\bm 1^\T\hat{\bm x}\hat{\bm y}^\T \bm 1 
    &\geq n^{\frac{rp}{r-p}+1}\mu^{\frac{rp}{r-p}}-Cn^{\frac{r}{r-p}}\sum_{i,j\in[n]}|  a^{\mathsf{la}}_{ij}|^{\frac{rp-r}{r-p}}-Cn^{\frac{rp-p}{r-p}}\sum_{i,j\in[n]}|  a^{\mathsf{la}}_{ij}|^{\frac{p}{r-p}}\\
    &\quad-C\sum_{i,j\in[n]}|  a^{\mathsf{la}}_{ij}|^{\frac{rp-r}{r-p}}\sum_{i,j\in[n]}| a^{\mathsf{la}}_{ij} |^{\frac{p}{r-p}}.
\end{split}
\end{align}
Furthermore, \eqref{hatxy} and \Cref{l:product0} imply that \whp,
\begin{align}\label{eq504}
\begin{split}
    \hat{\bm y}^\T  A^{\mathsf{int}}\hat{\bm x}=\left(n|\mu|\right)^{\frac{rp}{r-p}-1}\sgn(\mu)\bm 1^\T  A^{\mathsf{int}}\bm 1\geq -Cn^{\frac{rp}{r-p}-1}\sum_{i,j\in[n]}\left| a^{\mathsf{int}}_{ij}\right|.
\end{split}
\end{align}
For the last term in \eqref{eq501}, we first write it as
\begin{align}
\begin{split}
    \hat{\bm y}^\T  A^{\mathsf{la}}\hat{\bm x}=\sum_{i,j\in[n]} a^{\mathsf{la}}_{ij}\hat x_i\hat y_j,
\end{split}
\end{align}
and analyze the contribution of each summand by considering the following two cases:
\begin{enumerate}
\item When $ a^{\mathsf{la}}_{ij}=0$, the contribution is trivially $0$.
\item Suppose $ a^{\mathsf{la}}_{ij}\neq 0$. By \Cref{lemma:typicalBehaviour} item (a), \eqref{eta-para2} and \eqref{a-split}, we have \whp, $i\neq j$. We only consider the case $i<j$, the case $j<i$ can be analyzed in the same way with the same conclusion. By \Cref{lemma:typicalBehaviour} item (b), $ a^{\mathsf{la}}_{ij}$ is the only nonzero entry on $i$-th row and $j$-th column. Together with the symmetry of $ A^{\mathsf{la}}$, we have \whp, 
\begin{align}
\begin{split}
    \hat x_i = (n\left|\mu\right|)^{\frac{p}{r-p}}\sgn(\mu) + |  a^{\mathsf{la}}_{ij}|^{\frac{p}{r-p}}\sgn( a^{\mathsf{la}}_{ij}),\qquad\hat y_j=(n\left|\mu\right|)^{\frac{r(p-1)}{r-p}}+| a^{\mathsf{la}}_{ij} |^{\frac{r(p-1)}{r-p}}.
\end{split}
\end{align}
Therefore, it follows that 
\begin{align}
\begin{split}
     a^{\mathsf{la}}_{ij}x_iy_j \geq |  a^{\mathsf{la}}_{ij}|^{\frac{rp}{r-p}}-Cn^{\frac{p}{r-p}}| a^{\mathsf{la}}_{ij}|^{\frac{p(r-1)}{r-p}}-Cn^{\frac{r(p-1)}{r-p}}| a^{\mathsf{la}}_{ij}|^{\frac{r}{r-p}}-Cn^{\frac{rp}{r-p}-1}\left| a^{\mathsf{la}}_{ij}\right|.
\end{split}
\end{align}
\end{enumerate}

Combining the above two cases, we have \whp,
\begin{align}\label{eq503}
\begin{split}
    \hat{\bm y}^{\T} A^{\mathsf{la}}\hat{\bm x}&\geq \sum_{i,j\in[n]}\left| a^{\mathsf{la}}_{ij} \right|^{\frac{rp}{r-p}}-Cn^{\frac{p}{r-p}}\sum_{i,j\in[n]}| a^{\mathsf{la}}_{ij}|^{\frac{rp-p}{r-p}}-Cn^{\frac{rp-r}{r-p}}\sum_{i,j\in[n]}| a^{\mathsf{la}}_{ij}|^{\frac{r}{r-p}}\\
    &\quad-Cn^{\frac{rp}{r-p}-1}\sum_{i,j\in[n]}| a^{\mathsf{la}}_{ij}|.
\end{split}
\end{align}

Inserting \eqref{eq502}, \eqref{eq504} and \eqref{eq503} into \eqref{eq501} and recall that $\gamma=\frac{r-p}{rp}$, we have \whp,
\begin{align}\label{eq507}
\begin{split}
    \mathcal L_1(A)\geq n^{\frac{1}{\gamma}+1}\mu^{\frac1\gamma}+\sum_{i,j\in[n]}\left| a^{\mathsf{la}}_{ij} \right|^{\frac1\gamma}-\sum_{i=1}^7R_i',
\end{split}
\end{align}
where
\begin{alignat}{2}
    &R_1'\coloneqq Cn^{\frac{r}{r-p}}\sum_{i,j\in[n]}| a^{\mathsf{la}}_{ij}|^{\frac{rp-r}{r-p}}, & \quad &R_2'\coloneqq Cn^{\frac{rp-p}{r-p}}\sum_{i,j\in[n]}| a^{\mathsf{la}}_{ij}|^{\frac{p}{r-p}},\\
    &R_3'\coloneqq Cn^{\frac{p}{r-p}}\sum_{i,j\in[n]}| a^{\mathsf{la}}_{ij}|^{\frac{rp-p}{r-p}}, &\quad &R_4'\coloneqq Cn^{\frac{rp-r}{r-p}}\sum_{i,j\in[n]}| a^{\mathsf{la}}_{ij}|^{\frac{r}{r-p}},\\
    &R_5'\coloneqq \left(n\mu\right)^{\frac{rp}{r-p}-1}\sum_{i,j\in[n]}\left| a^{\mathsf{la}}_{ij}\right|,&\quad &R_6'\coloneqq \left(n\mu\right)^{\frac{rp}{r-p}-1}\sum_{i,j\in[n]}\left| a^{\mathsf{int}}_{ij}\right|,
\end{alignat}
and
\begin{equation}
R_7'\coloneqq C\sum_{i,j\in[n]}|  a^{\mathsf{la}}_{ij}|^{\frac{rp-r}{r-p}}\sum_{i,j\in[n]}| a^{\mathsf{la}}_{ij} |^{\frac{p}{r-p}}.
\end{equation}

We now claim that \whp,
\begin{align}\label{eq508}
\begin{split}
    R_i'=o\left(b_n^{\frac1\gamma}\right),\quad 1\leq i\leq 7.
\end{split}
\end{align}
This can be verified as follows:
\begin{enumerate}[(1)]
\item By the definition of $ a^{\mathsf{la}}_{ij}$ in \eqref{a-split} and \Cref{l:sum-big} with $q=\frac{rp-r}{r-p}$ and $\beta=\frac{2-\zeta}{\alpha}$, we have \whp
\begin{align}\label{eq505}
\begin{split}
    \sum_{i,j\in[n]}| a^{\mathsf{la}}_{ij}|^{\frac{rp-r}{r-p}}=\sum_{i,j\in[n]}|a_{ij}|^{\frac{rp-r}{r-p}}\bbm 1\left\{|a_{ij}|>n^{\frac{2-\zeta}{\alpha}}\right\}=O_\delta\left(n^{\frac{2}{\alpha}\frac{rp-r}{r-p}+\zeta}\right).
\end{split}
\end{align}
Hence, the choice of $\zeta$ in \eqref{eta-para2}, $\alpha<2$ and \eqref{eqn:order-bn} imply that \whp,
\begin{align}
\begin{split}
    R_1'= O_\delta\left(n^{\frac{2}{\alpha}\frac{rp-r}{r-p}+\frac{r}{r-p}+\zeta}\right)=O_{\delta}\left(n^{\frac{2}{\alpha\gamma}+\left(2-\frac{2}{\alpha}-\frac{\alpha}{2}\right)\frac{p-1}{r}}\right) =o\left(b_n^{\frac1\gamma}\right).
\end{split}
\end{align}
The estimation of $R_2',R_3',R_4',R_5'$ and $R_7'$ are similar.
\item By the definition of $ a^{\mathsf{int}}_{ij}$ in \eqref{a-split} and \Cref{l:sum-big} with $q=1$ and $\beta=\frac{1+\eta}{\alpha}$, we have \whp,
\begin{align}
\begin{split}
    \sum_{i,j\in[n]}| a^{\mathsf{int}}_{ij}|=\sum_{i,j\in[n]}|a_{ij}|\bbm 1\left\{|a_{ij}|\geq n^{\frac{1+\eta}{\alpha}} \right\}=O_\delta\left(n^{\frac{1}{\alpha}+1}\right).
\end{split}
\end{align}
Hence, \eqref{rtop3-cond}, together with \eqref{eqn:order-bn}, implies that \whp,
\begin{align}
\begin{split}
    R_6'=O_\delta\left(n^{\frac{1}{\alpha}+1}\right)=o\left(b_n^{\frac1\gamma}\right).
\end{split}
\end{align}
\end{enumerate}
Combining \eqref{rp-sum-small}, \eqref{eq507} and \eqref{eq508}, we have established that \whp,
\begin{align}\label{l1}
\begin{split}
    \mathcal L_1(A)\geq n^{\frac1\gamma+1}\mu^{\frac1\gamma}+\sum_{i,j\in[n]}|a_{ij}|^{\frac1\gamma}+o\left(b_n^{\frac1\gamma}\right).
\end{split}
\end{align}

\noindent\emph{Step 2: Estimation of $\mathcal L_2(A)$.} Fix $i\in[n]$. We estimate the $r$-th power of $\hat x_i$, considering two cases:
\begin{enumerate}
\item Suppose there exists $i^*\in[n]$ such that $ a^{\mathsf{la}}_{ii^*}\neq 0$. Then by \eqref{a-split} and the choice of $\zeta$ in \eqref{eta-para2}, we have \whp,
\begin{align}\label{eq511}
\begin{split}
    \left| a^{\mathsf{la}}_{ii^*}\right|^{\frac{p}{r-p}} \geq(n \mu)^{\frac{p}{r-p}}.
\end{split}
\end{align}
By \Cref{lemma:typicalBehaviour} (b), \eqref{eta-para2} and \eqref{a-split}, we have \whp, $ a^{\mathsf{la}}_{ij}=0$ for all $j\neq i^*$. 
Together with definition of $\hat x$ in \eqref{hatxy}, \eqref{eq511} and the elementary inequality $(x+y)^q\leq x^q+Cx^{q-1}y$ for all $x,q>0$ and $0\leq y\leq x$, this implies that \whp,
\begin{align}\label{eq509}
\begin{split}
    \left|\hat{x}_i\right|^r \leq\sum_{j\in[n]}| a^{\mathsf{la}}_{ij}|^{\frac{1}{\gamma}}+C n^{\frac{p}{r-p}}\sum_{j\in[n]}| a^{\mathsf{la}}_{ij}|^{\frac{p(r-1)}{r-p}} .
\end{split}
\end{align}
\item Suppose $ a^{\mathsf{la}}_{ij}=0$ for all $j\in[n]$. Then by definition of $\hat x$ in \eqref{hatxy}, we have
\begin{align}\label{eq510}
\begin{split}
    \left|\hat{x}_i\right|^r=(n \mu)^{\frac{1}{\gamma}} .
\end{split}
\end{align}
\end{enumerate}
Combining \eqref{eq509} and \eqref{eq510}, we have \whp,
\begin{align}\label{eq512}
\begin{split}
    \|\hat{\bm x}\|_r^r\leq n^{\frac{1}{\gamma}+1}\mu^{\frac{1}{\gamma}}+\sum_{i,j\in[n]}| a^{\mathsf{la}}_{ij}|^{\frac{1}{\gamma}}+Cn^{\frac{p}{r-p}}\sum_{i,j\in[n]}| a^{\mathsf{la}}_{ij}|^{\frac{p(r-1)}{r-p}}.
\end{split}
\end{align}
By \Cref{l:sum-big} with $q=\frac{p(r-1)}{r-p}$ and $\beta=\frac{2-\zeta}{\alpha}$, the choice of $\zeta$ in \eqref{eta-para2} and \eqref{eqn:order-bn}, we have \whp,
\begin{align}
\begin{split}
    n^{\frac{p}{r-p}}\sum_{i,j\in[n]}\left| a^{\mathsf{la}}_{ij}\right|^{\frac{p(r-1)}{r-p}}&=n^{\frac{p}{r-p}}\sum_{i, j \in[n]}\left|a_{i j}\right|^{\frac{p(r-1)}{r-p}} \bbm{1}\left\{\left|a_{i j}\right|>n^{\frac{2-\zeta}{\alpha}}\right\}\\
    &=O_\delta\left(n^{\frac{2}{\alpha} \frac{p(r-1)}{r-p}+\frac{p}{r-p}+\zeta}\right)\\
    &=O_\delta\left(n^{\frac{2}{\alpha\gamma}+\left(2-\frac{2}{\alpha}-\frac{\alpha}{2}\right)\frac{p-1}{r}}\right)\\
    &=o\left(b_n^{\frac{1}{\gamma}}\right).
\end{split}
\end{align}
Together with \eqref{rp-sum-small} and \eqref{eq512}, this implies that \whp,
\begin{align}\label{eq514}
\begin{split}
    \|\hat{\bm x}\|_r^r\leq n^{\frac{1}{\gamma}+1} \mu^{\frac{1}{\gamma}}+\sum_{i,j\in[n]}|a_{ij}|^{\frac1\gamma}+o\left(b_n^{\frac{1}{\gamma}}\right) .
\end{split}
\end{align}
Similarly, we have \whp,
\begin{align}\label{eq515}
\begin{split}
    \|\hat{\bm y}\|_{p^*}^{p^*}\leq n^{\frac{1}{\gamma}+1} \mu^{\frac{1}{\gamma}}+\sum_{i,j\in[n]}|a_{ij}|^{\frac1\gamma}+o\left(b_n^{\frac{1}{\gamma}}\right) .
\end{split}
\end{align}
Together with \eqref{rp-sum}, the definition of $\mathcal L_2(A)$ in \eqref{l1-l2}, the elementary inequality $(x+y)^q\geq x^q+qx^{q-1}y$ for all $x>0,q\leq 0$ and $y\geq 0$, and the fact that $-\frac{1}{r}-\frac{1}{p^*}=\gamma-1\leq 0$, the last two displays imply that \whp, we have
\begin{align}\label{l2}
\begin{split}
    \frac{1}{\mathcal L_2(A)}&\geq \left(n^{\frac{1}{\gamma}+1} \mu^{\frac{1}{\gamma}}+\sum_{i,j\in[n]}|a_{ij}|^{\frac1\gamma}+o\left(b_n^{\frac{1}{\gamma}}\right)\right)^{\gamma-1}\\
    &\geq n^{-\frac1\gamma+\gamma}\mu^{-\frac1\gamma+1}+\left(\gamma-1\right) n^{-\frac{2}{\gamma}+\gamma-1}\mu^{-\frac{2}{\gamma}+1}\sum_{i,j\in[n]}|a_{ij}|^{\frac1\gamma}+o\left(n^{-\frac{2}{\gamma}+\gamma-1}b_n^{\frac1\gamma}\right).
\end{split}
\end{align} 
To complete the proof, combine \eqref{rp-sum}, 
\eqref{l1} and \eqref{l2} with the definition of $\mathcal L(A)$ in \eqref{d:lower}.
\end{proof}
\subsubsection{Intuition behind the lower bound}
\label{rem:lbound}

We now provide some intuition into the lower bound in \Cref{l:lower}.    
Although the matrix of interest is $\mu\bm 1\bm 1^\top + A^{\mathsf{int}}+ A^{\mathsf{la}}$, for simplicity we focus only on the largest entries and the centering, and discuss a lower bound for the simpler modified matrix $ A^{\mathsf{la}}_{\mu,n}\coloneqq |\mu|\bm 1\bm 1^\top +| A^{\mathsf{la}}|$.  Since this matrix is nonnegative, Boyd's nonlinear power method can be used to arrive at an approximation to its $r \to p$ norm. 
Indeed, recall the definitions of $\Psi_q$ in \eqref{def:psi}, and define the operator $S:\mathbb{R}^n \to \mathbb{R}^n$ by 
\begin{equation}
\begin{aligned}
    S\bm x\coloneqq \Psi_{r^*}( A^{\mathsf{la}}_{\mu,n}\Psi_p( A^{\mathsf{la}}_{\mu,n}\bm x)).
\end{aligned}
\end{equation}
By the nonlinear power method, if $\bm v^{(0)}$ has all positive entries, and we define $\bm v^{(k+1)}\coloneqq S\bm v^{(k)}$, then it is known that (see \cite[Section~4.1]{dhara2020r})
\begin{equation}\label{power}
\begin{aligned}
    \lim_{k\rightarrow\infty}\frac{\left\| A^{\mathsf{la}}_{\mu,n}\bm v^{(k)}\right\|_p}{\left\|\bm v^{(k)}\right\|_r} =\left\| A^{\mathsf{la}}_{\mu,n}\right\|_{r\rightarrow p}.
\end{aligned}
\end{equation}
If $A_{\mu,n}$ were light-tailed, a natural choice for a near optimizer would  be the all one vector or one iteration of it under $S$,  as considered in  \cite{dhara2020r}.   However, this vector is far from optimal in the heavy-tailed setting, and it is {\em a priori} unclear how to identify a near optimizer.   Towards this end, 
we leverage two simple observations. First, by the definition of $ A^{\mathsf{la}}$ in \eqref{a-split} and the choice of $\zeta$ in \eqref{eta-para2}, we have the following property,
\begin{equation}\label{fact1}
\begin{aligned}
    \text{for each $i\in[n]$, we either have $(| A^{\mathsf{la}}|\bm 1)_i=0$ or $\mu n\ll (| A^{\mathsf{la}}|\bm 1)_i$}.
\end{aligned}
\end{equation}
Second, by the symmetry of $ A^{\mathsf{la}}$ and the fact that, \whp, each row of $ A^{\mathsf{la}}$ has at most one nonzero entry (as a consequence of \Cref{lemma:typicalBehaviour} (b)), for all $q>0$, we have
\begin{equation}\label{fact2}
\begin{aligned}
    | A^{\mathsf{la}}|(| A^{\mathsf{la}}|\bm 1)^q = (| A^{\mathsf{la}}|\bm 1)^{q+1}.
\end{aligned}
\end{equation}
Starting from an initial condition of the form $\bm v^{[k]}\coloneqq(\mu n)^{k}+\left(| A^{\mathsf{la}}|\bm 1\right)^{k}$ for some $k\geq 0$, and then using \eqref{fact1} and \eqref{fact2} repeatedly, it is not hard to see that 
\begin{equation}
\begin{aligned}
    S\bm v^{[k]} \approx (n|\mu|)^{\frac{p}{r-1}+\frac{(p-1)k}{r-1}}+\left(| A^{\mathsf{la}}|\bm 1\right)^{\frac{p}{r-1}+\frac{(p-1)k}{r-1}}=\bm v^{\left[\frac{p}{r-1}+\frac{(p-1)k}{r-1} \right]}.
\end{aligned}
\end{equation}
Therefore, the operation of $S$ on this class of vectors effectively acts as the transformation  
$k\mapsto \frac{p}{r-1}+\frac{(p-1)k}{r-1}$ 
on the exponents of each of the two terms. 
It is easy to see that the iteration of such a  map converges to a unique fixed point of the form $\bm v^{[k_*]}$, where $k_* = \frac{p}{r-p}$.
In view of \eqref{power}, it is natural to choose the near optimizer to be equal to this limit: 
\begin{equation}
\begin{aligned}
    \hat{\bm x} \coloneqq (n|\mu|)^{\frac{p}{r-p}}+\left(| A^{\mathsf{la}}|\bm 1\right)^{\frac{p}{r-p}}.
\end{aligned}
\end{equation}
Then the dual formulation \eqref{eq:rtop-dual} of the $r \to p$ norm suggests that $\hat{\bm y}$ should be set equal to 
\begin{align}\label{determine-y}
\begin{split}
    \argmax_{\bm y } \left\langle \bm y, A^{\mathsf{la}}_{\mu,n}\hat{\bm x}\right\rangle. 
\end{split}
\end{align}
Although this vector can be explicitly determined using (the equality case) of H\"older's inequality, we in fact choose $\hat{\bm y}$ in \eqref{hatxy} to be a suitable perturbation of the actual optimizer in \eqref{determine-y} 
can be  shown to be close to the actual optimizer using \eqref{fact1}, and is more convenient for 
 subsequent calculations. 
\subsection{Asymptotic equivalence of upper and lower bounds}\label{ss:matching}
\begin{proof}[Proof of \Cref{thm:rtop3}]
Comparing \Cref{l:upper} and \Cref{l:lower}, we have shown that
\begin{align}\label{matching}
\begin{split}
    \|\mu\bm 1\bm 1^\T+ A^{\mathsf{int}}+ A^{\mathsf{la}}\|_{r\rightarrow p}=n^{1+\gamma}|\mu|+\gamma n^{\gamma-\frac1\gamma}|\mu|^{1-\frac1\gamma}\sum_{i,j\in[n]}\left|a_{ij}\right|^{\frac1\gamma}+o\left(n^{\gamma-\frac1\gamma} b_n^{\frac1\gamma}\right).
\end{split}
\end{align}
Note that condition \eqref{rtop3-cond} implies \eqref{rtop2-cond2}, and in particular, \eqref{later-claim} holds \whp.
Combining \eqref{later-claim} and \eqref{rtop3-cond}, we have \whp,
\begin{align}\label{sss}
\begin{split}
    \| A^{\mathsf{sm}}\|_{r\rightarrow p}=o\left(n^{\gamma-\frac1\gamma} b_n^{\frac1\gamma}\right).
\end{split}
\end{align}
Combining \eqref{matching} and \eqref{sss}, we have \whp
\begin{align}\label{match-all}
\begin{split}
    \|\mu\bm 1\bm 1^\T+A\|_{r\rightarrow p}=n^{1+\gamma}|\mu|+\gamma n^{\gamma-\frac1\gamma}|\mu|^{1-\frac1\gamma}\sum_{i,j\in[n]}\left|a_{ij}\right|^{\frac1\gamma}+o\left(n^{\gamma-\frac1\gamma} b_n^{\frac1\gamma}\right).
\end{split}
\end{align}
Finally, by \cite[Theorem~3.8.2]{Durrett19} and \eqref{rtop3-cond},
\begin{align}\label{fluct-part}
\begin{split}
    b_n^{-\frac{1}{\gamma}}\sum_{i,j\in[n]}\left|a_{ij}\right|^{\frac1\gamma}\xrightarrow{(d)} Y_{\alpha\gamma},
\end{split}
\end{align}
where $Y_{\alpha\gamma}$ is an $\alpha\gamma$-stable random variable. Together, \eqref{match-all} and \eqref{fluct-part} imply the desired result.
\end{proof}

\section{Proofs related to the \texorpdfstring{$\ell_r$}{l-r}-Levy-Grothendieck problem }\label{s:gro}
\subsection{Proof of \texorpdfstring{\Cref{thm:gro1}}{Theorem 2.4}}\label{ss:gro1}
\begin{proof}[Proof of \Cref{thm:gro1}]
Let $a_{*}\coloneqq \max_{i,j}|a_{i,j}|$ and let $i_*,j_*\in[n]$ be the (random) location of the maximal entry of $A$, that is $a_*=|a_{i_*,j_*}|$. 
From \eqref{eqn:a_*Dist}, we have that $a_*=\Theta_\delta(b_n)$ \whp. By \Cref{lemma:typicalBehaviour} (a), this implies $i_*\neq j_*$ \whp. Next, set $\mathbf x\coloneqq2^{-1/r}(\mathbf e_{i_*}+\sgn\left(a_*\right)\mathbf e_{j_*})$. Then by the symmetry of $A$, we have \whp,
\begin{align}\label{easy-lower}
\begin{split}
    M_r(A)\geq \mathbf x^\trans A\mathbf x=2^{1-\frac{2}{r}}a_*.
\end{split}
\end{align}
We now consider two cases.

\noindent\emph{Case 1.} 
Suppose $r\in(1,2]$.
Let $\mathbf v=(v_1,\ldots,v_n)\in\mathbb R^n$ with $\|\bm v\|_r\leq 1$ denote a maximizing vector for the optimization problem \ref{opt1} associated with $M_r(A)$, and let $v_*\coloneqq \max_{i\in[n]}|v_i|$. By the method of Langrange multipliers, it is easy to see that the following relationship must hold:
\begin{align}\label{gro-cond}
\begin{split}
    M_r(A)|v_i|^{r-1}=\left|\sum_{j=1}^na_{ij}v_j\right|.
\end{split}
\end{align}
We claim that \whp,
\begin{align}\label{entry-bound}
\begin{split}
    v_*\leq 2^{-1/r}(1+o(1)).
\end{split}
\end{align}
To see why the claim is true, let $k_*\in[n]$ satisfy $v_{*}=|v_{k_*}|$. By \eqref{gro-cond} we have
\begin{align}\label{gro-cond1}
\begin{split}
    M_r(A)v_*^{r-1}=M_r(A)|v_{k_*}|^{r-1}=\left|\sum_{j\in[n]}a_{k_*j}v_j\right|.
\end{split}
\end{align}
Let $\delta(\alpha)$ be as defined in \Cref{lemma:typicalBehaviour} (c). 
Suppose $\max_{j\in[n]}|a_{k_*j}|\leq b_n^{3/4+\delta(\alpha)}$. Then, since $\|\bm v\|_r\leq 1$ implies $\|\bm v\|_{\infty}\leq 1$, \Cref{lemma:typicalBehaviour} (c) combined with $\eqref{gro-cond1}$ yields
\begin{align}\label{v1}
\begin{split}
    M_r(A)v_*^{r-1}\leq \sum_{j=1}^n|a_{k_*j}|=o(b_n).
\end{split}
\end{align}
On the other hand, suppose there exists $k'\in[n]$ such that $|a_{k_*k'}|>b_n^{3/4+\delta(\alpha)}$. Then by properties (a) and (b) of \Cref{lemma:typicalBehaviour}, \whp, there exists a unique such $k'$, and further $k'\neq k_*$. Moreover, \Cref{lemma:typicalBehaviour} (c) combined with \eqref{gro-cond1} and the constraint $\|\bm v\|_r\leq 1$, yields \whp
\begin{align}\label{v2}
\begin{split}
    M_r(A)v_{*}^{r-1}\leq |a_{k_*k'}||v_{k'}|+\sum_{j:\,|a_{k_*j}|\leq b_n^{3/4+\delta(\alpha)}}|a_{k_*j}|\leq \left(1-v_*^r\right)^{\frac{1}{r}}a_*\left(1+o(1)\right)+o(b_n) .
\end{split}
\end{align}
Combining \eqref{easy-lower}, \eqref{v1} and \eqref{v2}, we have, \whp
\begin{align}\label{v3}
\begin{split}
    2^{1-\frac{2}{r}}a_{*}v_*^{r-1}\leq M_r(A)v_*^{r-1}\leq \left(1-v_*^r\right)^{\frac{1}{r}}a_*\left(1+o(1)\right)+o(b_n).
\end{split}
\end{align}
Suppose that \whp, $v_*=1-o(1)$. Then \eqref{v3} and the relation $a_*=\Theta_\delta(b_n)$, which holds by \eqref{eqn:a_*Dist}, imply that \whp, $v_*=o(1)$, 
leading to a contradiction. So $v^*\neq 1-o(1)$. Together with \eqref{v3} and $a_*=\Theta_\delta\left(b_n\right)$, this implies that \whp,
\begin{align}
\begin{split}
    2^{1-\frac{2}{r}}v_*^{r-1}\leq (1-v_*^r)^{\frac{1}{r}}\left(1+o(1)\right),
\end{split}
\end{align}
which is equivalent to
\begin{align}\label{tivia-manip}
\begin{split}
    \frac{v_*^{r-1}}{\left(1-v_*^r\right)^{\frac{1}{r}}}\leq 2^{\frac{2}{r}-1}\left(1+o(1)\right).
\end{split}
\end{align}
Note that the map $\phi: x\mapsto \frac{x^{r-1}}{(1-x^r)^{1/r}}$ is increasing for $x\in(0,1)$, and $\phi(2^{-1/r})=2^{2/r-1}$. Together with \eqref{tivia-manip}, this completes the proof of claim \eqref{entry-bound}.

Combining \eqref{gro-cond1} and \eqref{eqn:sumOfTypicalRow}, we have \whp,
\begin{align}
\begin{split}
    M_r(A)v_*^{r-1}\leq v_*\sum_{j=1}^n|a_{k_*j}|\leq v_*a_*\left(1+o(1)\right).
\end{split}
\end{align}
Along with \eqref{entry-bound} and the fact that $r\leq 2$, this implies that \whp,
\begin{align}\label{boring-upper}
\begin{split}
    M_r(A)\leq v_*^{2-r}a_{*}\left(1+o(1)\right)\leq 2^{1-\frac{2}{r}}a_*(1+o(1)).
\end{split}
\end{align}
To complete the proof of $r\in(1,2]$, combine \eqref{easy-lower}, \eqref{boring-upper} and \eqref{eqn:a_*Dist}.

\noindent\emph{Case 2.} Suppose $r=1$. On the one hand, note that the lower bound \eqref{easy-lower} continues to hold. On the other hand, by definition $M_1(A)\leq \inf_{r\in (1,2]}M_r(A)$. Together with \eqref{boring-upper}, this implies that \whp,
\begin{align}\label{tricky-upper}
\begin{split}
    M_1(A)\leq \inf_{r\in(1,2]}M_r(A)=2^{-1}a_*(1+o(1)).
\end{split}
\end{align}
The result follows on combining \eqref{easy-lower}, \eqref{tricky-upper} and \eqref{eqn:a_*Dist}. 
\end{proof}
\subsection{Proof of \texorpdfstring{\Cref{thm:gro2}}{Theorem 2.6} and Theorem \ref{thm:gro2'}}\label{ss:gro2}
\begin{proof}[Proof of \Cref{thm:gro2} and \Cref{thm:gro2'}]
Fix $r\in(2,\infty]$. 
Recall the decomposition $A= A^{\mathsf{sm}}+ A^{\mathsf{int}}+ A^{\mathsf{la}}$ in \eqref{a-split}. Using the triangle inequality and \eqref{eq:rtop-dual}, we have
\begin{align}
\begin{split}
    \left|M_r(A)-M_r( A^{\mathsf{la}})\right|&\leq  \sup_{\|\bm x\|_r=1}\left(\left|\bm x^\T  A^{\mathsf{int}}\bm x \right|+\left|\bm x^\T A^{\mathsf{sm}}\bm x\right|\right)\\
    &\leq \sup_{\|\bm x\|_r=\|\bm y\|_r=1}\left(\bm y^\T A^{\mathsf{int}}\bm x+\bm y^\T A^{\mathsf{sm}}\bm x\right)\\
    &= \| A^{\mathsf{sm}}\|_{r\rightarrow r^*}+\| A^{\mathsf{int}}\|_{r\rightarrow r^*}.
\end{split}
\end{align}
Now set $p\coloneqq \frac{r}{r-1}=r^*$ and $\gamma\coloneqq \frac{r-2}{r}=\frac{1}{p}-\frac{1}{r}$. Then $p^*=r>p\geq 1$, and 
conditions \eqref{gro2-cond1} and \eqref{gro2-cond2} imply that $\gamma$ satisfies conditions \eqref{rtop2-cond1} and \eqref{rtop2-cond2}, respectively. Thus, the results in \Cref{s:rtop2} are valid for $r$ and $p=r^*$. In particular, together with the last display, \Cref{lem:omit-small} and \Cref{lem:omit-middle} imply that \whp, 
\begin{align}\label{small-diff-gro}
\begin{split}
    \left|M_r(A)-M_r( A^{\mathsf{la}})\right|=o(b_n),
\end{split}
\end{align}
under assumptions of \Cref{thm:gro2} or \Cref{thm:gro2'}.
Together with \Cref{lem:fluct-large}, this completes the proof.
\end{proof}
\subsection{Proof of \texorpdfstring{\Cref{thm:gro3}}{Theorem 2.8}}\label{ss:gro3}
\begin{proof}[Proof of \Cref{thm:gro3}]
Fix $r\in(2,\infty), \mu>0$ and set $p=r^*$.
Recall the decomposition of $A= A^{\mathsf{sm}}+ A^{\mathsf{int}}+ A^{\mathsf{la}}$ in \eqref{a-split} and let $ A^{\mathsf{la}}_{\mu}\coloneqq \mu\bm 1\bm 1^\T+ A^{\mathsf{int}}+ A^{\mathsf{la}}$. Note that \eqref{gro3-cond} implies \eqref{rtop2-cond2} with $\gamma=\frac{1}{p}-\frac{1}{r}=\frac{r-2}{r}$.
Hence, \eqref{later-claim} of \Cref{l:later-claim} and the condition $r>2$ imply that
\begin{align}\label{eq6000}
\begin{split}
    \|  A^{\mathsf{sm}}\|_{r\rightarrow r^*}=O_\delta\left(n^{\frac{1}{r^*}-\frac{1}{r}+\frac{1+\eta}{\alpha}}\right)=o\left(n^{\frac{r-2}{r}-\frac{r}{r-2}}b_n^{\frac{r}{r-2}}\right),
\end{split}
\end{align}
where the last equality uses \eqref{gro3-cond}.
With $A_\mu\coloneqq\mu\bm 1\bm 1^\T+A$, by the duality formula \eqref{eq:rtop-dual} and \eqref{eq6000}, we have \whp,
\begin{align}\label{eq600}
\begin{split}
    \left|M_r(A_\mu)-M_r\left( A^{\mathsf{la}}_\mu\right)\right|
    \leq \sup_{\|\bm x\|_r=1}\left|\bm x^\T  A^{\mathsf{sm}}\bm x\right|
    \leq \sup_{\|\bm x\|_r=\|\bm y\|_r=1}\bm y^\T A^{\mathsf{sm}}\bm x
    =\| A^{\mathsf{sm}}\|_{r\rightarrow p}
    =o\left(n^{\frac{r-2}{r}-\frac{r}{r-2}}b_n^{\frac{r}{r-2}}\right).
\end{split}
\end{align}
Using \eqref{ineq:dual}, we have
\begin{align}\label{eq601}
\begin{split}
    M_r\left( A^{\mathsf{la}}_\mu\right)\leq \| A^{\mathsf{la}}_\mu\|_{r\rightarrow r^*}.
\end{split}
\end{align}
Since \eqref{gro3-cond} implies \eqref{rtop3-cond} with $\gamma=\frac{1}{p}-\frac{1}{r}=\frac{r-2}{r}$, by \Cref{l:upper} and \eqref{eq601}, we have \whp,
\begin{align}\label{gro-upper}
\begin{split}
    M_r\left( A^{\mathsf{la}}_\mu\right)\leq n^{2-\frac{2}{r}} \mu+\frac{r-2}{r} \mu^{-\frac{2}{r-2}} n^{\frac{r-2}{r}-\frac{r}{r-2}} \sum_{i, j=1}^n\left|a_{i j}\right|^{\frac{r}{r-2}}+o\left(n^{\frac{r-2}{r}-\frac{r}{r-2}}b_n^{\frac{r}{r-2}}\right).
\end{split}
\end{align}
Recall the definition of $\hat{\bm x},\hat{\bm y}$ in \eqref{hatxy}. When $p=r^*$, we have $\hat{\bm x}=\hat{\bm y}$, and hence,
\begin{align}\label{eq602}
\begin{split}
    M_r\left( A^{\mathsf{la}}_\mu\right)\geq \frac{\hat{\bm x}^\T  A^{\mathsf{la}}_\mu\hat{\bm x}}{\|\hat{\bm x}\|_r^2}.
\end{split}
\end{align}
Combining \eqref{eq602}, the definition of $\mathcal L(A)$ in \eqref{d:lower} and \Cref{l:lower1}, we have \whp,
\begin{align}\label{gro-lower}
\begin{split}
    M_r\left( A^{\mathsf{la}}_\mu\right)\geq n^{2-\frac{2}{r}} \mu+\frac{r-2}{r} \mu^{-\frac{2}{r-2}} n^{\frac{r-2}{r}-\frac{r}{r-2}} \sum_{i, j=1}^n\left|a_{i j}\right|^{\frac{r}{r-2}}+o\left(n^{\frac{r-2}{r}-\frac{r}{r-2}}b_n^{\frac{r}{r-2}}\right).
\end{split}
\end{align}
Combining \eqref{eq600}, \eqref{gro-upper} and \eqref{gro-lower}, we have \whp,
\begin{align}\label{gro-match}
\begin{split}
    M_r(A_\mu)=n^{2-\frac{2}{r}} \mu+\frac{r-2}{r} \mu^{-\frac{2}{r-2}} n^{\frac{r-2}{r}-\frac{r}{r-2}} \sum_{i, j=1}^n\left|a_{i j}\right|^{\frac{r}{r-2}}+o\left(n^{\frac{r-2}{r}-\frac{r}{r-2}}b_n^{\frac{r}{r-2}}\right)
\end{split}
\end{align}
Finally, \cite[Theorem~3.8.2]{Durrett19} implies that
\begin{align}
\begin{split}
    \sum_{i,j\in[n]}\left|a_{ij}\right|^{\frac{r}{r-2}}\xrightarrow{(d)} Y_{\frac{\alpha(r-2)}{r}},
\end{split}
\end{align}
where $Y_{\frac{\alpha(r-2)}{r}}$ is an $\frac{\alpha(r-2)}{r}$-stable random variable. Together with \eqref{gro-match}, this completes the proof.
\end{proof}
\subsection{Proof of \texorpdfstring{\Cref{cor:ground}}{the ground energy}}\label{ss:spin}
\begin{proof}[Proof of \Cref{cor:ground}]Fix $\alpha\in(0,1)$. Then condition \eqref{gro2-cond1} of \Cref{thm:gro2} is satisfied and hence,  we conclude from that theorem that 
\begin{align}\label{implication}
\begin{split}
    b_n^{-1}\sup_{\|\bm x\|_\infty\leq 1}\bm x^\T A_n\bm x\xrightarrow{(d)} Y_\alpha,
\end{split}
\end{align}
where $Y_\alpha$ is an $\alpha$-stable random variable. Let $\Lambda_n\in\mathbb R^{n\times n}$ denote the diagonal of $A_n$. Then we have \whp,
\begin{align}\label{diag-small}
\begin{split}
    \|\Lambda\|_{\infty\rightarrow 1}=\sup_{\|\bm x\|_\infty=\|\bm y\|_{\infty}=1} \sum_{i\in[n]}a_{ii}x_iy_i\leq \sum_{i\in[n]}\left|a_{ii}\right|=O_\delta\left(n^{\frac{1}{\alpha}}\right),
\end{split}
\end{align}
where the last equality uses \cite[Theorem~3.8.2]{Durrett19}. 
Define $\mathring A_n\coloneqq A_n-\Lambda_n$. Combining \eqref{implication}, \eqref{diag-small} and \eqref{eqn:order-bn}, we have
\begin{align}\label{diag-free}
\begin{split}
    b_n^{-1}\sup_{\|\bm x\|_\infty\leq 1}\bm x^\T \mathring A_n\bm x\xrightarrow{(d)}Y.
\end{split}
\end{align}

Now we argue that the supremum in \eqref{diag-free} must be obtained at the extremal points $\left\{-1,1\right\}^n$ of the $n$-dimensional hypercube. To see this, pick a point $\bm x\in\R^n$ with $\|\bm x\|_\infty\leq 1$. 
Since $\mathring A_n$ is diagonal free, the map $\bm x\mapsto \bm x^\T \mathring A_n\bm x$ is linear in $x_i$ for each $i\in[n]$. Hence, we can iteratively replace $x_i$ by either $+1$ or $-1$ without decreasing $\bm x^\T\mathring A_n\bm x$ and the terminal point of this iterative procedure would be a point in $\{-1,1\}^n$.
Given the above argument, \eqref{diag-free} implies that
\begin{align}\label{extreme}
\begin{split}
    b_n^{-1}\sup_{\bm x\in\left\{-1,1\right\}^n}\bm x^\T \mathring A_n\bm x\xrightarrow{(d)}Y.
\end{split}
\end{align}
Once again invoking \eqref{diag-small} and \eqref{eqn:order-bn}, we conclude that
\begin{align}
\begin{split}
    b_n^{-1}\sup_{\bm x\in\left\{-1,1\right\}^n}\bm x^\T A_n\bm x\xrightarrow{(d)}Y.
\end{split}
\end{align}
This completes the proof.
\end{proof}

\bibliographystyle{amsplain}
\bibliography{project_r-p.bib}

\end{document}